\documentclass[11pt]{amsart}

\usepackage[margin=1in]{geometry} 
\usepackage{amsmath, amsthm, amsfonts, amssymb} 
\usepackage{mathrsfs} 
\usepackage{enumerate} 
\usepackage{xcolor} 
\usepackage{bbm} 
\usepackage{hyperref} 

\usepackage{chngcntr}
\counterwithin{equation}{section} 

\theoremstyle{theorem}
	
	\newtheorem{thm}{Theorem}[section]
	\newtheorem{lem}[thm]{Lemma}
	\newtheorem{prop}[thm]{Proposition}
	\newtheorem{cor}[thm]{Corollary}
	\newtheorem{conj}[thm]{Conjecture}
\theoremstyle{definition}
	\newtheorem{defn}[thm]{Definition}
	\newtheorem{rem}[thm]{Remark}
	\newtheorem{quest}[thm]{Question}

\newcommand{\N}{\mathbb{N}}
\newcommand{\Z}{\mathbb{Z}}
\newcommand{\Q}{\mathbb{Q}}
\newcommand{\R}{\mathbb{R}}
\newcommand{\C}{\mathbb{C}}
\newcommand{\T}{\mathbb{T}}

\newcommand{\F}{\mathbb{F}}

\newcommand{\B}{\mathcal{B}}

\newcommand{\bfX}{\mathbf{X}}
\newcommand{\bfY}{\mathbf{Y}}
\newcommand{\bfZ}{\mathbf{Z}}
\newcommand{\bfW}{\mathbf{W}}

\newcommand{\mX}{\mathcal{X}}
\newcommand{\mY}{\mathcal{Y}}
\newcommand{\mZ}{\mathcal{Z}}
\newcommand{\mI}{\mathcal{I}}
\newcommand{\mW}{\mathcal{W}}

\newcommand{\M}{\mathcal{M}}

\newcommand{\eps}{\varepsilon}

\renewcommand{\tilde}{\widetilde}
\renewcommand{\hat}{\widehat}

\newcommand{\tendsto}[1]{\xrightarrow[#1]{}}

\newcommand{\UClim}{\textup{UC-}\lim}

\newcommand{\id}{\textup{id}}

\newcommand{\ind}{\mathbbm{1}}

\newcommand{\E}[2]{\mathbb{E} \left[ #1 \mid #2 \right]}

\newcommand{\innprod}[2]{\left\langle #1, #2 \right\rangle}
\newcommand{\norm}[2]{\left\| #2 \right\|_{#1}}
\newcommand{\seminorm}[2]{{\left\vert\kern-0.25ex\left\vert\kern-0.25ex\left\vert #2 
    \right\vert\kern-0.25ex\right\vert\kern-0.25ex\right\vert}_{#1}}

\newcommand{\Hil}{\mathcal{H}}

\newcommand{\End}{\textup{End}}

\title{Khintchine-type double recurrence in abelian groups}

\author{Ethan Ackelsberg}
\address{School of Mathematics, Institute for Advanced Study, Princeton, NJ 08540}
\email{eackelsberg@ias.edu}

\date{\today}



\begin{document}

\maketitle


\begin{abstract}
	We prove a Khintchine-type recurrence theorem for pairs of endomorphisms of a countable discrete abelian group.
	As a special case of the main result, if $\Gamma$ is a countable discrete abelian group, $\varphi, \psi \in \End(\Gamma)$,
	and $\psi - \varphi$ is an injective endomorphism with finite index image, then for any ergodic measure-preserving
	$\Gamma$-system $\left( X, \mX, \mu, (T_g)_{g \in \Gamma} \right)$, any measurable set $A \in \mX$,
	and any $\eps > 0$, the set of $g \in \Gamma$ for which
	\begin{equation*}
		\mu \left( A \cap T_{\varphi(g)}^{-1} A \cap T_{\psi(g)}^{-1} A \right) > \mu(A)^3 - \eps
	\end{equation*}
	is syndetic.
	This generalizes the main results of \cite{abs}
	and essentially answers a question left open in that paper (\cite[Question 1.12]{abs}).
	
	For the group $\Gamma = \Z^d$, we deduce that for any matrices $M_1, M_2 \in M_{d \times d}(\Z)$
	whose difference $M_2 - M_1$ is nonsingular, any ergodic measure-preserving
	$\Z^d$-system $\left( X, \mX, \mu, (T_{\vec{n}})_{\vec{n} \in \Z^d} \right)$, any measurable set $A \in \mX$,
	and any $\eps > 0$, the set of $\vec{n} \in \Z^d$ for which
	\begin{equation*}
		\mu \left( A \cap T_{M_1 \vec{n}}^{-1} A \cap T_{M_2 \vec{n}}^{-1} A \right) > \mu(A)^3 - \eps
	\end{equation*}
	is syndetic,
	a result that was previously known only in the case $d = 2$ (see \cite[Theorem 7.1]{abs}).
	
	The key ingredients in the proof are:
	(1) a recent result obtained jointly with Bergelson and Shalom \cite{abs}
	that says that the relevant ergodic averages are controlled by a characteristic factor
	closely related to the \emph{quasi-affine} (or \emph{Conze--Lesigne}) factor;
	(2) an extension trick to reduce to systems with well-behaved (with respect to $\varphi$ and $\psi$) discrete spectrum; and
	(3) a description of Mackey groups associated to quasi-affine cocycles over rotational systems
	with well-behaved discrete spectrum.
\end{abstract}


\section{Introduction}


This paper is a continuation of work of the author together with Bergelson and Best \cite{abb} and Bergelson and Shalom \cite{abs} investigating the phenomenon of multiple recurrence with large intersections for actions of countable abelian groups (see also \cite{shalom}).


\subsection{Background and motivation}

The impetus for studying large intersections for multiple recurrence comes from the following two classical results:

\begin{thm}[Khintchine's recurrence theorem \cite{Kh}] \label{thm: khintchine}
	For any invertible measure-preserving system $\left(X, \mX, \mu, T \right)$, any $A \in \mX$, and any $\eps > 0$, the set
	\begin{equation*}
		\left\{ n \in \Z : \mu \left( A \cap T^{-n}A \right) > \mu(A)^2 - \eps \right\}
	\end{equation*}
	has bounded gaps.
\end{thm}

\begin{thm}[Furstenberg's multiple recurrence theorem \cite{furstenberg}] \label{thm: mult rec}
	For any invertible measure-preserving system $(X, \mX, \mu, T)$, any $A \in \mX$ with $\mu(A) > 0$, and any positive integer $k \in \N$,
	\begin{equation*}
		\liminf_{N-M \to \infty}{\frac{1}{N-M} \sum_{n=M}^{N-1}{\mu\left( A \cap T^{-n}A \cap \cdots \cap T^{-kn}A \right)}} > 0.
	\end{equation*}
	In particular, there exists $c > 0$ such that the set
	\begin{equation*}
		\left\{ n \in \Z : \mu\left( A \cap T^{-n}A \cap \cdots \cap T^{-kn}A \right) > c \right\}
	\end{equation*}
	has bounded gaps.
\end{thm}

A subset of $\Z$ with bounded gaps is called \emph{syndetic}.
More generally, in a countable discrete abelian group $(\Gamma, +)$, a subset $S \subseteq \Gamma$ is \emph{syndetic} if finitely many translates of $S$ cover $\Gamma$.

With the aim of finding a common refinement of Theorems \ref{thm: khintchine} and \ref{thm: mult rec},
Bergelson, Host, and Kra \cite{bhk} asked whether, for a measure-preserving system $(X, \mX, \mu, T)$, a set $A \in \mX$ with $\mu(A) > 0$, and $\eps > 0$, the set
\begin{equation*}
	\left\{ n \in \Z : \mu \left( A \cap T^{-n}A \cap \dots \cap T^{-kn}A \right) > \mu(A)^{k+1} - \eps \right\}
\end{equation*}
is syndetic.
They found that the answer depends on the length $(k+1)$ of the arithmetic progression:

\begin{thm}[\cite{bhk}, Theorems 1.2 and 1.3] \label{thm: bhk} ~
	\begin{enumerate}[(1)]
		\item	For any ergodic measure-preserving system $\left(X, \mX, \mu, T \right)$,
			any $A \in \mX$, and any $\eps > 0$, the sets
			\begin{equation*}
				\left\{ n \in \Z : \mu \left( A \cap T^{-n}A \cap T^{-2n}A \right) > \mu(A)^3 - \eps \right\}
			\end{equation*}
			and
			\begin{equation*}
				\left\{ n \in \Z : \mu \left( A \cap T^{-n}A \cap T^{-2n}A \cap T^{-3n}A \right) > \mu(A)^4 - \eps \right\}
			\end{equation*}
			are syndetic.
		\item	There exists an ergodic measure-preserving system $(X, \mX, \mu, T)$ with the following property.
			For any $l \in \N$, there exists $A = A(l) \in \mX$ with $\mu(A) > 0$ such that
			\begin{equation*}
				\mu \left( A \cap T^{-n}A \cap T^{-2n}A \cap T^{-3n}A \cap T^{-4n}A \right) \le \mu(A)^l
			\end{equation*}
			for every $n \ne 0$.
	\end{enumerate}
\end{thm}

\begin{rem}
	The ergodicity assumption in item (1) cannot be dropped.
	An adaptation of Behrend's construction of sets avoiding 3-term arithmetic progressions \cite{behrend}
	can be used to produce a counterexample for the non-ergodic transformation $T(x,y) = (x, y+x)$ on the 2-torus $\T^2$;
	see \cite[Theorem 2.1]{bhk}.
\end{rem}

The combinatorial content of Theorem \ref{thm: bhk}(1) is expressed by the following closely related result:

\begin{thm}[\cite{green}, Theorem 1.10; \cite{gt}, Theorem 1.12] \label{thm: gt}
	Let $\alpha, \eps > 0$.
	\begin{enumerate}
		\item	There exists $N_3 = N_3(\alpha, \eps) \in \N$ such that
			if $N \ge N_3$ and $A \subseteq \{1, \dots, N\}$ has cardinality $|A| \ge \alpha N$,
			then there exists $d \in \N$ such that
			\begin{equation*}
				\left| \left\{ a \in \N : \{a, a+d, a+2d\} \subseteq A \right\} \right| > (\alpha^3 - \eps) N.
			\end{equation*}
		\item	There exists $N_4 = N_4(\alpha, \eps) \in \N$ such that
			if $N \ge N_4$ and $A \subseteq \{1, \dots, N\}$ has cardinality $|A| \ge \alpha N$,
			then there exists $d \in \N$ such that
			\begin{equation*}
				\left| \left\{ a \in \N : \{a, a+d, a+2d, a+3d\} \subseteq A \right\} \right| > (\alpha^4 - \eps) N.
			\end{equation*}
	\end{enumerate}
\end{thm}

\begin{rem}
	(1) The positive integers $d \in \N$ appearing in Theorem \ref{thm: gt} are sometimes referred to as
	\emph{popular differences}, since they are common differences for many arithmetic progressions contained in $A$.
	
	(2) Theorem \ref{thm: bhk} can be converted directly into a combinatorial statement involving sets of positive upper Banach density by a version of the Furstenberg correspondence principle that produces ergodic measure-preserving systems; see \cite[Section 1.2]{bhk}.
	However, no simple argument is known to deduce Theorem \ref{thm: gt} from Theorem \ref{thm: bhk} or \emph{vice versa}.
\end{rem}

In other contexts in which a multiple recurrence result is known,
one may again ask whether it is possible to find a corresponding Khintchine-type enhancement.
Pursuing this line of inquiry, Bergelson, Tao, and Ziegler \cite{btz2} established a Khintchine-type recurrence result
for actions of the group $\F_p^{\infty}$:

\begin{thm}[\cite{btz2}, Theorems 1.12 and 1.13] \label{thm: btz} ~
	Fix a prime $p$ and $a, b \in \F_p$.
	For any ergodic measure-preserving $\F_p^{\infty}$-system $\left( X, \mX, \mu, (T_g)_{g \in \F_p^{\infty}} \right)$,
	any $A \in \mX$, and any $\eps > 0$, the sets
	\begin{equation*}
		\left\{ g \in \F_p^{\infty} : \mu \left( A \cap T_{ag}^{-1}A \cap T_{bg}^{-1}A \right) > \mu(A)^3 - \eps \right\}
	\end{equation*}
	and
	\begin{equation*}
		\left\{ g \in \F_p^{\infty} : \mu \left( A \cap T_{ag}^{-1}A \cap T_{bg}^{-1}A \cap T_{(a+b)g}^{-1}A \right) > \mu(A)^4 - \eps \right\}
	\end{equation*}
	are syndetic.
\end{thm}

A finitary combinatorial analogue of Theorem \ref{thm: btz} along the lines of Theorem \ref{thm: gt}
can be deduced using the methods established in \cite{green} and \cite{gt}, which in fact apply to general finite abelian groups.
See also \cite[Lecture 4]{green-lectures}. \\

The most general multiple recurrence result with which we will concern ourselves is the following,
which can be seen as a consequence of \cite{fk-IP} or \cite[Theorem B]{austin}:

\begin{thm} \label{thm: endomorphism Szemeredi}
	Let $\Gamma$ be a countable discrete abelian group.
	Let $k \in \N$, and let $\varphi_1, \dots, \varphi_k \in \End(\Gamma)$.
	For any measure-preserving $\Gamma$-system $\left( X, \mX, \mu, (T_g)_{g \in \Gamma} \right)$
	and any set $A \in \mX$, the set
	\begin{equation*}
		\left\{ g \in \Gamma : \mu \left( A \cap T_{\varphi_1(g)}^{-1}A \cap \dots \cap T_{\varphi_k(g)}^{-1}A \right) > 0 \right\}
	\end{equation*}
	is syndetic.
\end{thm}

\begin{rem}
	When dealing with topological groups, one may wish to impose various continuity assumptions (for instance, on the endomorphisms $\varphi_1, \dots, \varphi_k$ or on the action of $\Gamma$ on $(X, \mX, \mu)$).
	Moreover, notions of largeness for subsets of $\Gamma$ such as syndeticity and upper Banach density (discussed below in Section \ref{sec: combinatorics}) depend on the topology on $\Gamma$.
	We assume that $\Gamma$ is discrete in order to avoid such topological issues.
\end{rem}

The foregoing discussion motivates the following general definition:

\begin{defn}
	Let $\Gamma$ be a countable discrete abelian group.
	A family of endomorphisms $\varphi_1, \dots, \varphi_k \in \End(\Gamma)$ has the \emph{large intersections property}
	if the following holds:
	for any ergodic measure-preserving $\Gamma$-system $\left( X, \mX, \mu, (T_g)_{g \in \Gamma} \right)$,
	any $A \in \mX$ and any $\eps > 0$, the set
	\begin{equation*}
		\left\{ g \in \Gamma : \mu \left( A \cap T_{\varphi_1(g)}^{-1}A \cap \dots \cap T_{\varphi_k(g)}^{-1}A \right) > \mu(A)^{k+1} - \eps \right\}
	\end{equation*}
	is syndetic.
\end{defn}

We now give a brief summary of the previously known results about the large intersections property
in general countable discrete abelian groups.

In \cite{abb}, a far-reaching generalization of Theorems \ref{thm: bhk} and \ref{thm: btz}
for configurations of length 3 and 4 was obtained (in a slight abuse of notation,
we abbreviate a family of endomorphisms of the form $\{g \mapsto a_1 g, \dots, g \mapsto a_k g\}$ by $\{a_1, \dots, a_k\}$):

\begin{thm}[\cite{abb}, Theorems 1.10 and 1.11] \label{thm: abb}
	Let $\Gamma$ be a countable discrete abelian group.
	\begin{enumerate}[(1)]
		\item	If $\varphi, \psi \in \End(\Gamma)$ are such that all three subgroups
			$\varphi(\Gamma)$, $\psi(\Gamma)$, and $(\psi - \varphi)(\Gamma)$ have finite index
			in $\Gamma$, then $\{\varphi, \psi\}$ has the large intersections property.
		\item	If $a, b \in \Z$ and all four subgroups $a\Gamma$, $b\Gamma$, $(a+b)\Gamma$, and $(b-a)\Gamma$
			have finite index in $\Gamma$, then $\{a, b, a+b\}$ has the large intersections property.
	\end{enumerate}
\end{thm}

\begin{rem}
	Endomorphisms of groups with finite index conditions of the kind appearing in item (1) of Theorem \ref{thm: abb}
	have led to fruitful developments in a number of areas of ergodic theory and combinatorics;
	see, e.g., \cite{griesmer, prendiville, ll, gll, pilatte}.
\end{rem}

Item (2) in Theorem \ref{thm: abb} was also obtained independently by Shalom; see \cite[Theorem 1.3]{shalom}.
In joint work with Bergelson and Shalom, item (1) of Theorem \ref{thm: abb} was strengthened as follows:

\begin{thm}[\cite{abs}, Theorems 1.11 and 1.13] \label{thm: abs}
	Let $\Gamma$ be a countable discrete abelian group.
	\begin{enumerate}[(1)]
		\item	Suppose $\varphi, \psi \in \End(\Gamma)$ and two of the three subgroups $\varphi(\Gamma)$, $\psi(\Gamma)$,
			and $(\psi - \varphi)(\Gamma)$ have finite index in $\Gamma$.
			Then $\{\varphi, \psi\}$ has the large intersections property.
		\item	Suppose $a, b \in \Z$ are distinct, nonzero integers such that $(b-a)\Gamma$ has finite index in $\Gamma$.
			Then $\{a, b\}$ has the large intersections property.
	\end{enumerate}
\end{thm}

This result leaves the following as a natural open question:

\begin{quest}[\cite{abs}, Question 1.12] \label{quest: motivation}
	Let $\Gamma$ be a countable discrete abelian group.
	Suppose $\varphi, \psi \in \End(\Gamma)$ such that $(\psi - \varphi)(\Gamma)$ has finite index in $\Gamma$.
	Does $\{\varphi, \psi\}$ have the large intersections property?
\end{quest}

\begin{rem}
	There are a variety of examples of pairs $\{\varphi, \psi\}$ without the large intersections property
	(see \cite[Example 10.2]{abb} and \cite[Theorem 1.14]{abs}),
	so it is necessary to impose some condition on $\varphi, \psi \in \End(\Gamma)$
	(such as the finite index assumption in Question \ref{quest: motivation})
	in order to hope for the large intersections property.
\end{rem}

The goal of this paper is to extend the techniques in \cite{abs} to answer Question \ref{quest: motivation} affirmatively
under a mild additional technical assumption.
As we will see, this condition is always satisfied when the endomorphisms are obtained as multiplication by integers
or when the group $\Gamma$ is equal to $\Z^d$ for some $d \in \N$.
Hence, we are able to reproduce Theorem \ref{thm: abs}(2)
and fully resolve Question \ref{quest: motivation} for $\Gamma = \Z^d$.


\subsection{Main results}

Our main result is the following:

\begin{thm} \label{thm: main}
	Let $\Gamma$ be a countable discrete abelian group.
	Let $\varphi, \psi \in \End(\Gamma)$.
	Suppose there exist endomorphisms $\eta, \varphi', \psi', \theta_1, \theta_2 \in \End(\Gamma)$ such that:
	\begin{enumerate}[(i)]
		\item	$\eta(\Gamma)$ is a finite index subgroup of $\Gamma$;
		\item	$\varphi = \varphi' \circ \eta$ and $\psi = \psi' \circ \eta$;
		\item	$\theta_1 \circ \varphi' + \theta_2 \circ \psi'$ is injective; and
		\item	$(\psi' - \varphi')(\Gamma)$ is a finite index subgroup of $\Gamma$.
	\end{enumerate}
	Then for any ergodic measure-preserving $\Gamma$-system $\left( X, \mX, \mu, (T_g)_{g \in \Gamma} \right)$,
	any $A \in \mX$, and any $\eps > 0$, the set
	\begin{equation*}
		\left\{ g \in \Gamma : \mu \left( A \cap T_{\varphi(g)}^{-1} A \cap T_{\psi(g)}^{-1}A \right) > \mu(A)^3 - \eps \right\}
	\end{equation*}
	is syndetic.
\end{thm}

\begin{rem}
	The conditions (i)--(iv) in Theorem \ref{thm: main} may be forbidding at first glance.
	We give a brief explanation here and also refer the reader to the special cases outlined below
	for developing stronger intuition about each condition.
	
	Conditions (i) and (ii) can be interpreted as follows.
	Define a new $\Gamma$-action on $(X, \mX, \mu)$ by $S_g = T_{\eta(g)}$.
	Then
	\begin{equation*}
		\mu \left( A \cap T_{\varphi(g)}^{-1}A \cap T_{\psi(g)}^{-1}A \right) = \mu \left( A \cap S_{\varphi'(g)}^{-1}A \cap S_{\psi'(g)}^{-1}A \right),
	\end{equation*}
	so it suffices to prove a statement about the system $\left( X, \mX, \mu, (S_g)_{g \in \Gamma} \right)$
	and the endomorphisms $\varphi', \psi' \in \End(\Gamma)$.
	The key consequence of (i) is that the new system $\left( X, \mX, \mu, (S_g)_{g \in \Gamma} \right)$
	has only finitely many ergodic components, and our dynamical tools are flexible enough to handle this situation.
	
	Condition (iii) imposes a level of nontriviality to the pair $\{\varphi', \psi'\}$
	by ensuring that the map $g \mapsto (\varphi'(g), \psi'(g))$ is injective.
	This condition also turns out to be crucial to describing the limiting behavior of double ergodic averages
	associated with $\{\varphi', \psi'\}$.
	
	Finally, condition (iv) is the key assumption in order to get started with analyzing the relevant double ergodic average
	by invoking \cite[Theorem 4.10]{abs}, which is proved using the van der Corput differencing trick
	(see Lemma \ref{lem: vdC} below).
	It is absolutely essential to the method used in this paper,
	though it is less clear whether condition (iv) is needed to obtain the desired conclusion.
	A concrete example where we do not know whether or not the large intersections property holds
	(and for which condition (iv) does not hold) is the following.
	Let $\Gamma = (\Q_{>0}, \cdot)$ be the group of positive rational numbers under multiplication (considered as a discrete group).
	Let $a, b \in \N$ be coprime.
	It was asked in \cite[Question 1.18]{abs} whether the pair $\{q \mapsto q^a, q \mapsto q^b\}$
	has the large intersections property, and we are presently unable to make any substantial progress on this question.
	For this example, although condition (iv) does not hold, it is nevertheless the case that
	the group generated by $\{q^a : q \in \Q_{>0}\}$ and $\{q^b : q \in \Q_{>0}\}$ has finite index in $\Q_{>0}$
	(in fact, it is equal to $\Q_{>0}$), which eliminates many of the possible approaches to producing a counterexample.
	Additional discussion of the difficulties involved in this problem can be found in \cite[Section 2.7]{abs}.
\end{rem}

We now turn to several consequences of Theorem \ref{thm: main}.

Theorem \ref{thm: main} includes Theorem \ref{thm: abs}(2) as a special case.
Given $a, b \in \Z$ such that $(b-a)\Gamma$ has finite index in $\Gamma$,
let $d = \gcd(a,b)$, $a' = \frac{a}{d}$, and $b' = \frac{b}{d}$.
Since $d \mid b-a$, we have $d\Gamma \supseteq (b-a)\Gamma$, so $d \Gamma$ has finite index in $\Gamma$.
The integers $a'$ and $b'$ are coprime, so there exist $c_1, c_2 \in \Z$ such that $c_1a' + c_2b' = 1$.
Finally $(b' - a') \mid (b-a)$, so $(b' - a')\Gamma \supseteq (b-a)\Gamma$ has finite index in $\Gamma$.
Taking $\varphi(g) = ag$, $\psi(g) = bg$, $\eta(g) = dg$, $\varphi'(g) = a'g$, $\psi'(g) = b'g$,
$\theta_1(g) = c_1g$, and $\theta_2(g) = c_2g$ and applying Theorem \ref{thm: main}
reproduces the conclusion of Theorem \ref{thm: abs}(2).

Another illustrative special case of Theorem \ref{thm: main} is the following:

\begin{cor} \label{cor: main cor}
	Let $\Gamma$ be a countable discrete abelian group.
	Let $\varphi, \psi \in \End(\Gamma)$ such that $\psi - \varphi$ is injective with finite index image.
	Then for any ergodic measure-preserving $\Gamma$-system $\left( X, \mX, \mu, (T_g)_{g \in \Gamma} \right)$,
	any $A \in \mX$, and any $\eps > 0$, the set
	\begin{equation*}
		\left\{ g \in \Gamma : \mu \left( A \cap T_{\varphi(g)}^{-1}A \cap T_{\psi(g)}^{-1}A \right) > \mu(A)^3 - \eps \right\}
	\end{equation*}
	is syndetic.
\end{cor}

For the group $\Gamma = \Z^d$, Corollary \ref{cor: main cor} takes the following shape:

\begin{cor} \label{cor: Z^d}
	Let $\bfX = \left( X, \mX, \mu, (T_{\vec{n}})_{\vec{n} \in \Z^d} \right)$ be an ergodic $\Z^d$-system.
	Then for any integer matrices $M_1, M_2 \in M_{d \times d}(\Z)$ such that $M_2 - M_1$ is nonsingular,
	any $A \in \mX$, and any $\eps > 0$, the set
	\begin{equation*}
		\left\{ \vec{n} \in \Z^d : \mu \left( A \cap T_{M_1\vec{n}}^{-1}A \cap T_{M_2\vec{n}}^{-1}A \right) > \mu(A)^3 - \eps \right\}
	\end{equation*}
	is syndetic.
\end{cor}

In the case $d = 2$, Corollary \ref{cor: Z^d} was established in \cite{abs} using a combination of different methods.
If $M_1$ is also nonsingular, then the conclusion follows from \cite[Theorem 1.11]{abs},
which is proved using methods similar to the current paper involving characteristic factors for multiple ergodic averages.
In the case that both $M_1$ and $M_2$ are rank one matrices, basic linear algebra combined with the Fubini property
for uniform Ces\`{a}ro limits (see \cite[Lemma 1.1]{bl-fubini}) allows one to prove a stronger result
that implies, in particular that the set
\begin{equation*}
	\left\{ \vec{n} \in \Z^d : \mu \left( A \cap T_{M_1\vec{n}}^{-1}A \cap T_{M_2\vec{n}}^{-1}A \right) > \mu(A)^3 - \eps \right\}
\end{equation*}
is syndetic for any (not necessarily ergodic) measure-preserving system
$\left( X, \mX, \mu, (T_{\vec{n}})_{\vec{n} \in \Z^2} \right)$, any $A \in \mX$, and any $\eps > 0$; see \cite[Theorem 7.1]{abs}.

The method for handling the case where both $M_1$ nor $M_2$ are singular matrices
does not easily generalize for $d \ge 3$.
Instead, we produce a new proof avoiding any use of matrix manipulations
that unifies the two different cases in order to apply to general $d \in \N$
and in fact to general countable discrete abelian groups.


\subsection{Combinatorial consequences and questions} \label{sec: combinatorics}

Recurrence results in ergodic theory translate into combinatorial statements about sets of positive density.
Let us first make precise what we mean by the density of a subset of an abelian group.
A \emph{F{\o}lner sequence} in a countable discrete abelian group $\Gamma$ is a sequence $(\Phi_N)_{N \in \N}$
of finite subsets of $\Gamma$ such that for any $x \in \Gamma$,
\begin{equation*}
	\frac{\left| (\Phi_N + x) \triangle \Phi_N \right|}{|\Phi_N|} \tendsto{N \to \infty} 0.
\end{equation*}
The \emph{upper density of a set $E \subseteq \Gamma$ along a F{\o}lner sequence $\Phi = (\Phi_N)_{N \in \N}$} is the quantity
\begin{equation*}
	\overline{d}_{\Phi}(E) = \limsup_{N \to \infty} \frac{\left| E \cap \Phi_N \right|}{|\Phi_N|}.
\end{equation*}
The \emph{upper Banach density} of $E \subseteq \Gamma$ is $d^*(E) = \sup_{\Phi} \overline{d}_{\Phi}(E)$, where the supremum is over all F{\o}lner sequences in $\Gamma$.
An immediate consequence of Theorem \ref{thm: main} together with a version of the Furstenberg correspondence principle for ergodic systems (see \cite[Theorem 2.8]{bf}) is the following:

\begin{thm}
	Let $\Gamma$ be a countable discrete abelian group.
	Let $\varphi, \psi \in \End(\Gamma)$ be as in Theorem \ref{thm: main}.
	Then for any $E \subseteq \Gamma$ and any $\eps > 0$, the set
	\begin{equation*}
		\left\{ g \in \Gamma : d^* \left( E \cap (E - \varphi(g)) \cap (E - \psi(g)) \right) > d^*(E)^3 - \eps \right\}
	\end{equation*}
	is syndetic.
\end{thm}

This strongly suggests that various finitary combinatorial results hold.
Namely, we conjecture that the following finitary version of Corollary \ref{cor: main cor} is true:

\begin{conj} \label{conj: abelian groups}
	For $\alpha, \eps > 0$, there exists $N_0 = N_0(\alpha, \eps)$ with the property:
	for any finite abelian group $G$ of order $N \ge N_0$,
	any $\varphi, \psi \in \End(G)$ such that $\psi - \varphi$ is an automorphism,
	and any set $A \subseteq G$ with $|A| \ge \alpha N$,
	there exists $y \in G \setminus \{0\}$ such that
	\begin{equation*}
		\left| \left\{ x \in G : \{x, x + \varphi(y), x + \psi(y)\} \subseteq A \right\} \right| > (\alpha^3 - \eps) N.
	\end{equation*}
\end{conj}

A natural conjecture in the setting of $\Z^d$, building on Corollary \ref{cor: Z^d}, is:

\begin{conj} \label{conj: Z^d}
	Let $M_1, M_2 \in M_{d \times d}(\Z)$ such that $M_2 - M_1$ is nonsingular.
	For any $\alpha, \eps > 0$, there exists $N_0 = N_0(\alpha, \eps, M_1, M_2) \in \N$ with the property:
	for any $N \ge N_0$ and any set $A \subseteq \{1, \dots, N\}^d$ with $|A| \ge \alpha N^d$,
	there exists $\vec{y} \in \Z^d \setminus \{0\}$ such that
	\begin{equation*}
		\left| \left\{ \vec{x} \in \Z^d : \left\{ \vec{x}, \vec{x} + M_1 \vec{y}, \vec{x} + M_2 \vec{y} \right\} \subseteq A \right\} \right| > (\alpha^3 - \eps) N^d.
	\end{equation*}
\end{conj}

If one imposes the additional condition in Conjecture \ref{conj: abelian groups} that $\varphi$ and $\psi$ are automorphisms,
then the conjecture is known to be true by \cite[Theorem 7.3]{bsst}.
Similarly, if the matrices $M_1$, $M_2$, and $M_2 - M_1$ in Conjecture \ref{conj: Z^d} are all nonsingular, then
the conclusion holds by \cite[Theorem 1.1]{bsst}.
(For the particular matrices
\begin{equation*}
	M_1 = \left( \begin{array}{cc} 1 & 0 \\ 0 & 1 \end{array} \right) \qquad \text{and} \qquad
	M_2 = \left( \begin{array}{cc} 0 & -1 \\ 1 & 0 \end{array} \right),
\end{equation*}
this was also shown by \cite[Theorem 1]{kovac}.)

The ergodic theoretic methods used in this paper are not immediately applicable in the finitary setting.
In order to resolve Conjectures \ref{conj: abelian groups} and \ref{conj: Z^d},
one should replace the dynamical tools with suitable analogues from higher order Fourier analysis.


\subsection{Outline of the paper}

The structure of the paper is as follows.
Section \ref{sec: prelim} is preparatory, collecting the relevant background material
that will be used in the proof of Theorem \ref{thm: main}.
The main technical results appear in Sections \ref{sec: extensions} and \ref{sec: Mackey},
where we prove the existence of extensions in which the Kronecker and quasi-affine factors
interact nicely with a fixed pair of endomorphisms $\{\varphi, \psi\}$.
We then prove a formula for the limit of double ergodic averages associated with $\{\varphi, \psi\}$ in Section \ref{sec: limit}.
Finally, we complete the proof of Theorem \ref{thm: main} in Section \ref{sec: proof}.


\section{Preliminaries} \label{sec: prelim}


\subsection{Uniform Ces\`{a}ro limits and the van der Corput differencing lemma}

Just as Furstenberg's multiple recurrence theorem (Theorem \ref{thm: mult rec}) establishes a recurrence result
by working with a multiple ergodic average,
we will prove Theorem \ref{thm: main} by studying an associated double ergodic average.
A sequence $(v_g)_{g \in \Gamma}$ in a (real or complex) topological vector space $V$ has \emph{uniform Ces\`{a}ro limit} equal to $v \in V$, denoted $\UClim_{g \in \Gamma}{v_g} = v$, if for any F{\o}lner sequence $(\Phi_N)_{N \in \N}$ in $\Gamma$, one has
\begin{equation*}
	\frac{1}{|\Phi_N|} \sum_{g \in \Phi_N}{v_g} \tendsto{N \to \infty} v.
\end{equation*}
In the group $\Gamma = \Z$, the uniform Ces\`{a}ro limit corresponds to the limit of averages appearing in Theorem \ref{thm: mult rec}, i.e.
\begin{equation*}
	\UClim_{n \in \Z}{v_n} = \lim_{N - M \to \infty}{\frac{1}{N-M} \sum_{n=M}^{N-1}{v_n}}
\end{equation*}

One of the main tools for handling uniform Ces\`{a}ro limits is the following version of the van der Corput differencing lemma:

\begin{lem}[cf. \cite{abb}, Lemma 2.2] \label{lem: vdC}
	Let $\Gamma$ be a countable discrete abelian group,
	and let $(u_g)_{g \in \Gamma}$ be a bounded sequence in a Hilbert space $\Hil$.
	If
	\begin{equation*}
		\xi_h := \UClim_{g \in \Gamma}{\innprod{u_{g+h}}{u_g}}
	\end{equation*}
	exists for every $h \in \Gamma$ and $\UClim_{h \in \Gamma}{\xi_h} = 0$,
	then $\UClim_{g \in \Gamma}{u_g} = 0$.
\end{lem}


\subsection{Host--Kra factors}

Let $\Gamma$ be a countable discrete abelian group,
and let $\bfX = \left( X, \mX, \mu, (T_g)_{g \in \Gamma} \right)$ be a measure-preserving $\Gamma$-system.
A \emph{factor} of $\bfX$ is a $(T_g)_{g \in \Gamma}$-invariant sub-$\sigma$-algebra $\mY \subseteq \mX$.
We may also refer to the system $\bfY = \left( X, \mY, \mu|_{\mY}, (T_g)_{g \in \Gamma} \right)$,
or any system isomorphic to $\bfY$, as a factor of $\bfX$.

The most important family of factors for our consideration is the family of \emph{Host--Kra factors}.
These factors are defined in terms of a family of seminorms, known as the \emph{Host--Kra seminorms} \cite{hk},
which are an ergodic-theoretic counterpart to the Gowers uniformity norms \cite{gowers} in additive combinatorics.

Let $\bfX = \left( X, \mX, \mu, (T_g)_{g \in \Gamma} \right)$ be a measure-preserving $\Gamma$-system.
For $g \in \Gamma$ and $f : X \to \C$, define $\Delta_gf := \overline{f} \cdot T_g f$.
Then for $k \in \N$ and $g_1, \dots, g_k \in \Gamma$, we define $\Delta_{g_1, \dots, g_k}$ inductively by
$\Delta_{g_1, \dots, g_k}f := \Delta_{g_k} \left( \Delta_{g_1, \dots, g_{k-1}}f \right)$.
For $f \in L^{\infty}(\mu)$ and $k \in \N$, we define the \emph{Host--Kra seminorm of order $k$} by
\begin{equation*}
	\seminorm{U^k}{f}^{2^k} = \UClim_{(g_1, \dots, g_k) \in \Gamma^k}{\int_X{\Delta_{g_1, \dots, g_k}f~d\mu}}.
\end{equation*}

It is shown that $\seminorm{U^k}{\cdot}$ is indeed a seminorm for each $k \in \N$ in \cite[Appendix A]{btz1}.
The corresponding Host--Kra factors are guaranteed by the following proposition:

\begin{prop}[\cite{btz1}, Proposition 1.10]
	Let $\Gamma$ be a countable discrete abelian group, let $\bfX = \left( X, \mX, \mu, (T_g)_{g \in \Gamma} \right)$
	be a measure-preserving $\Gamma$-system, and let $k \ge 0$.
	There exists a factor $\mZ^k$ with the property that for every $f\in L^\infty(\mu)$, one has
	\begin{equation*}
		\seminorm{U^{k+1}}{f} = 0 \iff \E{f}{\mZ^k} = 0.
	\end{equation*}
\end{prop}

If $\bfX = \left( X, \mX, \mu, (T_g)_{g \in \Gamma} \right)$ is an ergodic system,
then the first several Host--Kra factors are as follows:
\begin{itemize}
	\item	$\mZ^0$ is the trivial factor consisting of null and co-null subsets of $X$.
		(If $\bfX$ is not ergodic, then $\mZ^0 = \mI$, the $\sigma$-algebra of all $(T_g)_{g \in \Gamma}$-invariant sets.)
	\item	$\mZ^1$ is the \emph{Kronecker factor}.
		This is the smallest $\sigma$-algebra with respect to which all eigenfunctions are measurable.
		As a measure-preserving system, $\bfZ^1$ is isomorphic to a rotational system.
		That is, there exists a compact abelian group $Z$ and a homomorphism $\alpha : \Gamma \to Z$ with dense image
		such that $\bfZ^1$ is isomorphic to the system $\bfZ = (Z, \alpha)$, where $Z$ is equipped with the Haar measure
		and $g \in \Gamma$ acts by the rotation $z \mapsto z + \alpha_g$.
		Because of its relationship to the Kronecker factor, we refer to any ergodic rotational system $\bfZ = (Z, \alpha)$
		as an \emph{ergodic Kronecker system}.
		Such systems are uniquely determined (up to isomorphism) by their \emph{discrete spectrum}
		(i.e., the group of eigenvalues),
		which is given by
		\begin{equation*}
			\Lambda = \left\{ \lambda \circ \alpha : \lambda \in \hat{Z} \right\}.
		\end{equation*}
		Moreover, the topological system underlying any ergodic Kronecker system is uniquely ergodic.
		See \cite[Section 2.4]{abb} for a more in-depth discussion of Kronecker systems
		in the context of actions of countable discrete abelian groups.
	\item	$\mZ^2$ is the \emph{quasi-affine} (or \emph{Conze--Lesigne}) factor.
		As a measure-preserving system, $\bfZ^2$ is isomorphic to a group extension of the Kronecker factor,
		$\bfZ^2 \cong \bfZ^1 \times_{\sigma} H$, where the cocycle $\sigma$ satisfies a certain functional equation
		known as the \emph{Conze--Lesigne equation}; see Definition \ref{defn: quasi-affine} below.
\end{itemize}

\begin{defn}
	Let $\Gamma$ be a countable discrete abelian group, and let $k \ge 0$.
	An ergodic measure-preserving $\Gamma$-system $\bfX = \left( X, \mX, \mu, (T_g)_{g \in \Gamma} \right)$
	is a \emph{system of order $k$} if $\mZ^k = \mX$.
\end{defn}

Systems of order $k$ have the following properties:

\begin{prop}[cf. \cite{hk}, Section 4.6] \label{prop: order k factors}
	Let $\Gamma$ be a countable discrete abelian group, $k \ge 0$,
	and $\bfX = \left( X, \mX, \mu, (T_g)_{g \in \Gamma} \right)$ an ergodic measure-preserving $\Gamma$-system.
	\begin{enumerate}[(1)]
		\item	The Host--Kra factor $\bfZ^k$ is an order $k$ system.
		\item	If $\bfX$ is an order $k$ system and $\bfY$ is a factor of $\bfX$, then $\bfY$ is again a system of order $k$.
		\item	If $\bfY$ is a system of order $k$ and a factor of $\bfX$, then $\bfY$ is a factor of the Host--Kra factor $\bfZ^k$.
	\end{enumerate}
\end{prop}


\subsection{Relatively independent joinings}

The proof of Theorem \ref{thm: main} requires that the Host--Kra factors have certain convenient properties
that they may not have in general systems.
In order to produce these desirable properties, we will work with an extension of the original system.
The key construction to that end is the relatively independent joining of systems with respect to a common factor.
Let $\bfX_1 = \left( X_1, \mX_1, \mu_1, \left( T_{1,g} \right)_{g \in \Gamma} \right)$
and $\bfX_2 = \left( X_2, \mX_2, \mu_2, \left( T_{2,g} \right)_{g \in \Gamma} \right)$
be measure-preserving $\Gamma$-systems.
Suppose $\bfY = \left( Y, \mY, \nu, (S_g)_{g \in \Gamma} \right)$ is another measure-preserving $\Gamma$-system
that arises as a factor of both of the systems $\bfX_1$ and $\bfX_2$,
say with factor maps $\pi_1 : X_1 \to Y$ and $\pi_2 : X_2 \to Y$.
The \emph{relatively indpendent joining} (or \emph{fiber product}) of $\bfX_1$ and $\bfX_2$ with respect to $\bfY$ is the system
\begin{equation*}
	\bfX_1 \times_{\bfY} \bfX_2 = \left( X_1 \times X_2, \mX_1 \otimes \mX_2, \mu_1 \times_{\bfY} \mu_2, \left( T_{1,g} \times T_{2,g} \right)_{g \in \Gamma} \right),
\end{equation*}
where the measure $\mu_1 \times_{\bfY} \mu_2$ is defined by the equation
\begin{equation*}
	\int_{X_1 \times X_2}{(f_1 \otimes f_2)~d(\mu_1 \times_{\bfY} \mu_2)} = \int_Y{\E{f_1}{\mY} \cdot \E{f_2}{\mY}~d\nu}.
\end{equation*}
Note that the measure $\mu_1 \times_{\bfY} \mu_2$ is supported on the set
\begin{equation*}
	X_1 \times_{\bfY} X_2 = \left\{ (x_1, x_2) \in X_1 \times X_2 : \pi_1(x_1) = \pi_2(x_2) \right\}.
\end{equation*}

The relatively independent joining construction allows us to take an extension of a Host--Kra factor of a given system
and turn it into an extension of the full system:

\begin{thm} \label{thm: HK extension}
	Let $\bfX = \left( X, \mX, \mu, (T_g)_{g \in \Gamma} \right)$ be an ergodic $\Gamma$-system.
	Let $k \in \N$.
	Suppose $\bfZ^k$ is the Host--Kra factor of $\bfX$ of order $k$.
	Given any ergodic order $k$ extension $\tilde{\bfZ}^k$ of $\bfZ^k$,
	there exists an ergodic extension $\tilde{\bfX}$ of $\bfX$ such that
	$\tilde{\bfZ}^k$ is the Host--Kra factor of $\tilde{\bfX}$ of order $k$.
\end{thm}
\begin{proof}
	Define $\tilde{\bfX}$ as the relatively independent joining $\tilde{\bfX} = \bfX \times_{\bfZ^k} \tilde{\bfZ}^k$.
	First, $\tilde{\bfZ}^k$ is an order $k$ factor of $\tilde{\bfX}$, so it is a factor of the Host--Kra factor of order $k$
	by Proposition \ref{prop: order k factors}(3).
	
	Conversely, we want to show that the Host--Kra factor of $\tilde{\bfX}$ of order $k$ is a factor of $\tilde{\bfZ}^k$.
	Let $(u_i)_{i \in \N}$ be an orthonormal basis in $L^2(\tilde{Z}^k)$.
	Suppose $f \in L^2(\tilde{X})$ is measurable with respect to the Host--Kra factor of order $k$.
	Expand $f$ in the basis $(u_i)_{i \in \N}$:
	\begin{equation*}
		f(x,y) = \sum_{i \in \N}{a_i(x) u_i(y)}.
	\end{equation*}
	Fix $i \in \N$.
	Since $\ind_X \otimes u_i$ is $\tilde{\mZ}^k$-measurable and hence measurable
	with respect to the Host--Kra factor of order $k$, the product $(\ind_X \otimes \overline{u_i}) f$
	remains measurable with respect to the Host--Kra factor of order $k$.
	Therefore, $\E{(\ind_X \otimes \overline{u_i})f}{\mX}$ is measurable with respect to $\mZ^k$
	by items (2) and (3) in Proposition \ref{prop: order k factors}.
	By direct computation, since $(u_j)_{j \in \N}$ is an orthonormal basis in $L^2(\tilde{Z}^k)$, we have
	\begin{equation*}
		\E{(\ind_X \otimes \overline{u_i})f}{\mX} = \sum_{j \in \N}{\E{a_j \otimes \overline{u_i}u_j}{\mX}} = a_i.
	\end{equation*}
	Thus, $a_i$ is $\mZ^k$-measurable for each $i \in \N$.
	That is, $a_i(x) = b_i(\pi_1(x))$ for some function $b_i : Z^k \to \C$, where $\pi_1 : X \to Z^k$ is the factor map.
	But, letting $\pi_2 : \tilde{Z}^k \to Z^k$ be the other factor map, one has $\pi_1(x) = \pi_2(y)$ for a.e.~$(x,y) \in \tilde{X}$.
	Hence,
	\begin{equation*}
		f(x,y) = \sum_{i \in \N}{a_i(x) u_i(y)} = \sum_{i \in \N}{b_i(\pi_1(x)) u_i(y)} = \sum_{i \in \N}{b_i(\pi_2(y)) u_i(y)}
	\end{equation*}
	is $\tilde{\mZ}^k$-measurable.
\end{proof}


\subsection{Hilbert space-valued functions and unique ergodicity}

Let $\Hil$ be a Hilbert space.
Given a compact metric space $X$, a probability measure $\mu$ on $X$, and a continuous function $F : X \to \Hil$,
one can define the integral $\int_X{F~d\mu}$ to be the element of $\Hil$ satisfying
\begin{equation*}
	\innprod{\int_X{F~d\mu}}{v} = \int_X{\innprod{F(x)}{v}~d\mu(x)}
\end{equation*}
for every $v \in \Hil$.

A characterizing property of uniquely ergodic systems is the following:
a topological system $\left( X, (T_g)_{g \in \Gamma} \right)$ is uniquely ergodic (with unique invariant measure $\mu$)
if and only if for any continuous function $F : X \to \C$ and any $x_0 \in X$,
\begin{equation*}
	\UClim_{g \in \Gamma}{F(T_gx_0)} = \int_X{F~d\mu}.
\end{equation*}

The following lemma shows that the same result holds for Hilbert space-valued functions:

\begin{lem}[cf. \cite{hk-CL}, Lemma 4.3] \label{lem: Hil-valued ET}
	Let $\left( X, (T_g)_{g \in \Gamma} \right)$ be a topological $\Gamma$-system.
	Let $\Hil$ be a nontrivial Hilbert space.
	Then $\left( X, (T_g)_{g \in \Gamma} \right)$ is uniquely ergodic (with unique invariant measure $\mu$) if and only if for any continuous function $F : X \to \Hil$ and any $x_0 \in X$,
	\begin{equation*}
		\UClim_{g \in \Gamma}{F(T_gx_0)} = \int_X{F~d\mu}
	\end{equation*}
	in $\Hil$.
\end{lem}
\begin{proof}
	Suppose that for any continuous function $F : X \to \Hil$ and any $x_0 \in X$,
	\begin{equation*}
		\UClim_{g \in \Gamma}{F(T_gx_0)} = \int_X{F~d\mu}
	\end{equation*}
	in $\Hil$.
	Since $\Hil$ is nontrivial, it contains a copy of $\C$, so this implies if $F : X \to \C$ is continuous and $x_0 \in X$, then
	\begin{equation*}
		\UClim_{g \in \Gamma}{F(T_gx_0)} = \int_X{F~d\mu}
	\end{equation*}
	in $\C$.
	Therefore, $(X, (T_g)_{g \in \Gamma})$ is uniquely ergodic.
	
	Now suppose $\left( X, (T_g)_{g \in \Gamma} \right)$ is uniquely ergodic with unique invariant measure $\mu$.
	Replacing $F$ by $F - \int_X{F~d\mu}$, we may assume without loss of generality that $\int_X{F~d\mu} = 0$ in $\Hil$.
	Put $u_g = F(T_gx_0) \in \Hil$.
	For any $h \in \Gamma$, the function $\varphi_h : X \to \C$
	defined by $\varphi_h(x) = \innprod{F(T_hx)}{F(x)}$ is continuous.
	Hence, by unique ergodicity of $(T_g)_{g \in \Gamma}$,
	\begin{equation*}
		\xi_h := \UClim_{g \in \Gamma}{\innprod{u_{g+h}}{u_g}} = \UClim_{g \in \Gamma}{\varphi_h(T_gx_0)} = \int_X{\varphi_h~d\mu}.
	\end{equation*}
	For $x \in X$, consider the function $\psi_x(y) = \innprod{F(y)}{F(x)}$.
	This is a continuous function from $X$ to $\C$, so by unique ergodicity of $(T_h)_{h \in \Gamma}$, we have
	\begin{equation*}
		\UClim_{h \in \Gamma}{\psi_x(T_hy)} = \int_X{\psi_x~d\mu} = \innprod{\int_X{F~d\mu}}{F(x)} = 0
	\end{equation*}
	for every $y \in X$.
	In particular, we may take $y = x$, in which case
	\begin{equation*}
		\UClim_{h \in \Gamma}{\varphi_h(x)} = \UClim_{h \in \Gamma}{\innprod{F(T_hx)}{F(x)}} = \UClim_{h \in \Gamma}{\psi_x(T_hx)} = 0.
	\end{equation*}
	Integrating over $x \in X$ and applying the dominated convergence theorem, it follows that
	\begin{equation*}
		\UClim_{h \in \Gamma}{\xi_h} = 0.
	\end{equation*}
	Thus, by Lemma \ref{lem: vdC}, $\UClim_{g \in \Gamma}{F(T_gx_0)} = 0$ in $\Hil$.
\end{proof}


\subsection{Cocycles and coboundaries}

An important construction in ergodic theory is that of group extensions.
For our purposes, we will need only to consider extensions by abelian groups, which creates various simplifications.
Suppose $\bfX = \left( X, \mX, \mu, (T_g)_{g \in \Gamma} \right)$ is an ergodic $\Gamma$-system.
Given a compact abelian group $(H,+)$ and a measurable function $\sigma : \Gamma \times X \to H$,
we can define the \emph{group extension of $\bfX$ by $H$ over $\sigma$} as the system
\begin{equation*}
	\bfX \times_{\sigma} H := \left( X \times H, \mX \otimes \B_H, \mu \times m_H, \left( T^{\sigma}_g \right)_{g \in \Gamma} \right),
\end{equation*}
where $T^{\sigma}_g(x,y) := (T_gx, y + \sigma_g(x))$.
In order to obtain a $\Gamma$-action in this manner, the function $\sigma$ must satisfy the \emph{cocycle} equation
\begin{equation*}
	\sigma_{g+h}(x) = \sigma_g(T_hx) + \sigma_h(x)
\end{equation*}
for every $g, h \in \Gamma$ and $\mu$-a.e. $x \in X$.

Given any measurable function $F : X \to H$, one may construct a cocycle
\begin{equation*}
	\Delta_gF(x) := F(T_gx) - F(x).
\end{equation*}
Such a cocycle is called a \emph{coboundary}.

Two cocycles $\sigma$ and $\tau$ are \emph{cohomologous}, denoted $\sigma \sim \tau$,
if their difference $\tau - \sigma$ is a coboundary.
The following result is an easy exercise:

\begin{prop}[cf. \cite{glasner}, Lemma 3.20]
	Let $\bfX = \left( X, \mX, \mu, (T_g)_{g \in \Gamma} \right)$ be a $\Gamma$-system.
	Suppose $H$ is a compact abelian group and $\sigma, \tau : \Gamma \times X \to H$
	are cocycles such that $\sigma \sim \tau$.
	Then $\bfX \times_{\sigma} H \cong \bfX \times_{\tau} H$.
\end{prop}

When analyzing a cocycle $\sigma : \Gamma \times X \to H$ taking values in a compact abelian group $H$,
it is often useful to consider the family of cocycles $\chi \circ \sigma : \Gamma \times X \to S^1$
given by composition with characters $\chi \in \hat{H}$.

The following lemma gives a criterion for checking that a cocycle taking values in $S^1$ is a coboundary,
when the base system is an action by rotations on a compact abelian group.

\begin{lem}[\cite{abb}, Proposition 7.12] \label{lem: coboundary}
	Let $\bfZ = (Z,\alpha)$ be a Kronecker system and $\sigma : \Gamma \times Z \to S^1$ a cocycle.
	The following are equivalent:
	\begin{enumerate}[(i)]
		\item	$\sigma$ is a coboundary;
		\item	for any sequence $(g_n)_{n \in \N}$ in $\Gamma$ such that $\alpha_{g_n} \to 0$ in $Z$,
			one has $\sigma_{g_n} \to 1$ in $L^2(Z)$.
	\end{enumerate}
\end{lem}

\begin{rem}
	We do not assume that the Kronecker system $(Z, \alpha)$ appearing in Lemma \ref{lem: coboundary} is ergodic.
	This will be important for some later applications, e.g., Theorem \ref{thm: Mackey gluing}.
\end{rem}

\begin{cor} \label{cor: coboundary continuous extension}
	Let $\bfZ = (Z,\alpha)$ be a Kronecker system, and suppose $\sigma : \Gamma \times Z \to S^1$ is a coboundary.
	Then there is a function $\omega : Z \times Z \to S^1$ such that
	$t \mapsto \omega(t, \cdot)$ is a continuous map from $Z$ to $L^2(Z)$ and
	\begin{equation*}
		\omega(\alpha_g, z) = \sigma_g(z).
	\end{equation*}
	If $\bfZ$ is ergodic, then $\omega$ is defined uniquely almost everywhere.
\end{cor}
\begin{proof}
	This follows from the proof of \cite[Proposition 7.12]{abb}.
	We give a different proof here.
	Since $\sigma$ is a coboundary, we may write $\sigma = \Delta F$ for some $F : Z \to S^1$.
	That is,
	\begin{equation*}
		\sigma_g(z) = \frac{F(z+\alpha_g)}{F(z)}.
	\end{equation*}
	We may then take $\omega(t,z) = \frac{F(z+t)}{F(z)}$.
\end{proof}

Two additional families of cocycles will play an important role in this paper.

\begin{defn} \label{defn: quasi-affine}
	Let $\Gamma$ be a countable discrete abelian group, and let $\bfZ = (Z, \alpha)$ be an ergodic Kronecker system.
	A cocycle $\sigma : \Gamma \times Z \to S^1$ is
	\begin{enumerate}[(1)]
		\item	\emph{cohomologous to a character} if there exists $\gamma \in \hat{\Gamma}$
			such that $\sigma_g(z) \sim \gamma(g)$.
			That is, there exists a measurable function $F : Z \to S^1$ such that
			\begin{equation*}
				\sigma_g(z) = \gamma(g) \frac{F(z + \alpha_g)}{F(z)}
			\end{equation*}
			for every $g \in \Gamma$ and a.e. $z \in Z$.
		\item	\emph{quasi-affine} (or a \emph{Conze--Lesigne} cocycle) if for every $t \in Z$, the cocycle
			\begin{equation*}
				\frac{\sigma_g(z+t)}{\sigma_g(z)}
			\end{equation*}
			is cohomologous to a character.
	\end{enumerate}
\end{defn}

\begin{rem}
	According to the definition above, a cocycle $\sigma : \Gamma \times Z \to S^1$ is quasi-affine if and only if for each $t \in Z$, there exists a measurable function $F_t : Z \to S^1$ and a character $\gamma_t \in \hat{\Gamma}$ such that
	\begin{equation*}
		\frac{\sigma_g(z+t)}{\sigma_g(z)} = \gamma_t(g) \frac{F_t(z + \alpha_g)}{F_t(z)}
	\end{equation*}
	for every $g \in \Gamma$ and a.e. $z \in Z$.
	One may additionally ensure that the maps $(t, z) \mapsto F_t(z)$ from $Z \times Z$ to $S^1$ and $t \mapsto \gamma_t$ from $Z$ to $\hat{\Gamma}$ are Borel measurable by \cite[Proposition 2]{lesigne} and \cite[Proposition 10.5]{fw}.
\end{rem}

Suppose $\bfZ = (Z, \alpha)$ is an ergodic Kronecker system, $H$ is a compact abelian group,
and $\sigma : \Gamma \times Z \to H$ is a cocycle.
If $\chi \circ \sigma$ is quasi-affine for every $\chi \in \hat{H}$, then the group extension
$\bfX \times_{\sigma} H$ is called a \emph{quasi-affine} or \emph{Conze--Lesigne} system.

We now give characterizations of cocycles that are cohomologous to a character or quasi-affine
in the same vein as Lemma \ref{lem: coboundary}:

\begin{lem}[\cite{abb}, Proposition 7.13] \label{lem: cohomologous to character}
	Let $\bfZ = (Z, \alpha)$ be an ergodic Kronecker system and $\sigma : \Gamma \times Z \to S^1$ a cocycle.
	The following are equivalent:
	\begin{enumerate}[(i)]
		\item	$\sigma$ is cohomologous to a character;
		\item	for every $t \in Z$,
			\begin{equation*}
				\frac{\sigma_g(z+t)}{\sigma_g(z)}
			\end{equation*}
			is a coboundary;
		\item	there is a Borel set $A \subseteq Z$ with $m_Z(A) > 0$ such that
			\begin{equation*}
				\frac{\sigma_g(z+t)}{\sigma_g(z)}
			\end{equation*}
			is a coboundary for every $t \in A$;
		\item	for any sequence $(g_n)_{n \in \N}$ in $\Gamma$ with $\alpha_{g_n} \to 0$ in $Z$,
			there is a sequence $(c_n)_{n \in \N}$ in $S^1$ such that $c_n \sigma_{g_n}(z) \to 1$ in $L^2(Z)$.
	\end{enumerate}
\end{lem}

\begin{lem}[\cite{abb}, Proposition 7.15] \label{lem: quasi-affine}
	Let $\bfZ = (Z, \alpha)$ be an ergodic Kronecker system and $\sigma : \Gamma \times Z \to S^1$ a cocycle.
	The following are equivalent:
	\begin{enumerate}[(i)]
		\item	$\sigma$ is quasi-affine;
		\item	there is a Borel set $A \subseteq Z$ with $m_Z(A) > 0$ such that
			\begin{equation*}
				\frac{\sigma_g(z+t)}{\sigma_g(z)}
			\end{equation*}
			is cohomologous to a character for every $t \in A$;
		\item	for any sequence $(g_n)_{n \in \N}$ in $\Gamma$ with $\alpha_{g_n} \to 0$ in $Z$,
			there are sequences $(c_n)_{n \in \N}$ in $S^1$ and $(\lambda_n)_{n \in \N}$ in $\hat{Z}$ such that
			$c_n \lambda_n(z) \sigma_{g_n}(z) \to 1$ in $L^2(Z)$.
	\end{enumerate}
\end{lem}


\subsection{Mackey groups} \label{sec: Mackey prelim}

Mackey groups play an essential role in this paper.
We review the definition and the basic properties of the Mackey group in this section.

Let $\bfX = \left( X, \mX, \mu, (T_g)_{g \in \Gamma} \right)$ be an ergodic $\Gamma$-system,
$H$ a compact abelian group, and $\sigma : \Gamma \times X \to H$ a cocycle.
The \emph{range} of $\sigma$ is the closed subgroup $G_{\sigma}$ generated by $\{\sigma_g(x) : g \in \Gamma, x \in X\}$.
The cocycle $\sigma$ is \emph{minimal} if there is no cohomologous cocycle $\tau$ with $G_{\tau} \subsetneq G_{\sigma}$.

\begin{prop}[cf. \cite{glasner}, Theorem 3.25]
	Any cocycle $\sigma : \Gamma \times X \to H$ is cohomologous to a minimal cocycle.
	Moreover, if two minimal cocycles are cohomologous, then they have the same range.
\end{prop}

The \emph{Mackey group associated to $\sigma$} is defined to be the range of a minimal cocycle cohomologous to $\sigma$.
Several important properties of Mackey groups are collected in the following proposition:

\begin{prop} \label{prop: Mackey properties}~
	\begin{enumerate}[(1)]
		\item	If $M$ is the Mackey group associated to $\sigma$, then the annihilator of $M$ is
			\begin{equation*}
				M^{\perp} = \left\{ \chi \in \hat{H} : \chi \circ \sigma~\text{is a coboundary} \right\}.
			\end{equation*}
		\item	If $\sigma$ and $\tau$ are cohomologous, then their Mackey groups are equal.
		\item	The system $\bfX \times_{\sigma} H$ is ergodic if and only if $\sigma$ is minimal with range $G_{\sigma} = H$.
		\item	If $f : X \times H \to \C$ is $(T^{\sigma}_g)_{g \in \Gamma}$-invariant,
			then $f(x, y+m) = f(x, y)$ for every $m \in M$ and $(\mu \times m_H)$-a.e. $(x,y) \in X \times H$.
	\end{enumerate}
\end{prop}
\begin{proof}
	For property (1), see \cite[Proposition 2.5]{hk-CL}.
	Properties (2) and (3) are proved in \cite[Theorem 3.25]{glasner}.
	
	In the case $\Gamma = \Z$, property (4) appears in \cite[Proposition 2.4]{hk-CL} (see also \cite[Proposition 7.1]{fw}).
	We give a quick proof for general $\Gamma$.
	Suppose $f : X \times H \to \C$ is $\left( T^{\sigma}_g \right)_{g \in \Gamma}$-invariant.
	We may expand $f$ as a Fourier series:
	\begin{equation*}
		f(x,y) = \sum_{\chi \in \hat{H}}{c_{\chi}(x) \chi(y)}.
	\end{equation*}
	Then
	\begin{equation*}
		\left( T^{\sigma}_g f \right) (x,y) = \sum_{\chi \in \hat{H}}{c_{\chi}(T_gx) \chi(y + \sigma_g(x))} = \sum_{\chi \in \hat{H}}{c_{\chi}(T_gx) \chi(\sigma_g(x)) \chi(y)}.
	\end{equation*}
	The assumption that $T^{\sigma}_gf = f$ then implies
	\begin{equation} \label{eq: chi cohomology}
		c_{\chi}(x) = c_{\chi}(T_gx) \chi(\sigma_g(x))
	\end{equation}
	for every $\chi \in \hat{H}$.
	Hence, if $c_{\chi} \ne 0$, then \eqref{eq: chi cohomology} expresses $\chi \circ \sigma$ as a coboundary.
	Thus, $c_{\chi} = 0$ for $\chi \notin M^{\perp}$.
	For any $m \in M$, we therefore have
	\begin{equation*}
		f(x,y+m) = \sum_{\chi \in M^{\perp}}{c_{\chi}(x) \chi(y + m)} = \sum_{\chi \in M^{\perp}}{c_{\chi}(x) \chi(y)} = f(x,y)
	\end{equation*}
	as claimed.
\end{proof}

For a compact abelian group $K$, let $\M(Z, K)$ be the space of measurable functions $Z \to K$
with the topology of convergence in measure.
One can show that a sequence $(f_n)_{n \in \N}$ in $\M(Z,K)$ converges
if and only if $\left( \chi \circ f_n \right)_{n \in \N}$ converges in $L^2(Z)$ for every $\chi \in \hat{K}$;
see, e.g., \cite[Lemma 7.28]{abb}.
The following is an easy consequence of Corollary \ref{cor: coboundary continuous extension}
combined with the description of the Mackey group in item (1) of Proposition \ref{prop: Mackey properties}:

\begin{prop} \label{prop: continuous extension mod Mackey}
	Let $\bfZ = (Z, \alpha)$ be an ergodic Kronecker system, $H$ a compact abelian group,
	and $\sigma : \Gamma \times Z \to H$ a cocycle with Mackey group $M \subseteq H$.
	Then there is a function $\omega : Z \times Z \to H/M$ such that $t \mapsto \omega(t, \cdot)$ is a continuous
	map from $Z$ to $\M(Z, H/M)$ and
	\begin{equation*}
		\omega(\alpha_g, z) \equiv \sigma_g(z) \pmod{M}.
	\end{equation*}
\end{prop}

\begin{rem}
	If the cocycle $\sigma$ is minimal, then it takes values in the Mackey group $M$, in which case $\omega = 0$.
\end{rem}

The Mackey group plays an important role in the analysis of ergodic averages, as demonstrated by the following result:

\begin{prop}[cf. \cite{abb}, Proposition 7.7] \label{prop: Mackey ET}
	Let $\bfX = \left( X, \mX, \mu, (T_g)_{g \in \Gamma} \right)$ be an ergodic $\Gamma$-system,
	$H$ a compact abelian group, and $\sigma : \Gamma \times X \to H$ a cocycle.
	Let $M$ be the Mackey group associated to $\sigma$.
	Let $f \in L^2(\mu \times m_H)$.
	If for every $\chi \in M^{\perp}$ and $\mu$-a.e. $x \in X$, one has
	\begin{equation*}
		\int_H{f(x,y) \chi(y)~dy} = 0,
	\end{equation*}
	then
	\begin{equation*}
		\UClim_{g \in \Gamma}{T^{\sigma}_gf} = 0
	\end{equation*}
	in $L^2(\mu \times m_H)$.
\end{prop}
\begin{proof}
	By the mean ergodic theorem, let $\tilde{f} = \UClim_{g \in \Gamma}{T^{\sigma}_gf}$.
	We want to show $\tilde{f} = 0$.
	Since $L^2(H, m_H)$ is spanned by characters, it suffices to show
	\begin{equation*}
		\int_H{\tilde{f}(x,y) \chi(y)~dy} = 0
	\end{equation*}
	for every $\chi \in \hat{H}$ and $\mu$-a.e. $x \in X$.
	We will prove this identity in cases, depending on whether or not $\chi$ annihilates the Mackey group $M$. \\
	
	\underline{Case 1}: $\chi \notin M^{\perp}$.
	
	By item (4) of Proposition \ref{prop: Mackey properties},
	for any $m \in M$ and $(\mu \times m_H)$-a.e. $(x,y) \in X \times H$, we have $\tilde{f}(x,y+m) = \tilde{f}(x,y)$.
	Therefore, for $\mu$-a.e. $x \in X$,
	\begin{equation*}
		\int_H{\tilde{f}(x,y) \chi(y)~dy} = \int_H{\tilde{f}(x,y+m) \chi(y+m)~dy} = \chi(m) \int_H{\tilde{f}(x,y) \chi(y)~dy}.
	\end{equation*}
	Taking $m \in M$ such that $\chi(m) \ne 1$, this implies $\int_H{\tilde{f}(x,y) \chi(y)~dy} = 0$ as claimed. \\
	
	\underline{Case 2}: $\chi \in M^{\perp}$.
	
	Note that
	\begin{equation*}
		\int_H{\tilde{f}(x,y) \chi(y)~dy} = \E{\tilde{f} \cdot (\ind \otimes \chi)}{\mX}(x).
	\end{equation*}
	By item (1) of Proposition \ref{prop: Mackey properties}, there is a measurable function $F : X \to S^1$ such that
	\begin{equation*}
		\chi(\sigma_g(x)) = F(T_gx) \overline{F}(x)
	\end{equation*}
	for $\mu$-a.e. $x \in X$.
	We then compute directly:
	\begin{align*}
		\E{\tilde{f} \cdot (\ind \otimes \chi)}{\mX}
		 & = \UClim_{g \in \Gamma}{\E{\left( T^{\sigma}_gf \right) \cdot (\ind \otimes \chi)}{\mX}} \\
		 & = \UClim_{g \in \Gamma}{\left( \overline{\chi} \circ \sigma_g \right)
		 \cdot \E{T^{\sigma}_g \left( f \cdot (\ind \otimes \chi) \right)}{\mX}} \\
		 & = F \cdot \UClim_{g \in \Gamma}{T_g \left( \overline{F} \cdot \E{f \cdot (\ind \otimes \chi)}{\mX} \right)} \\
		 & = 0.
	\end{align*}
	In the last step, we have used the hypothesis $\E{f \cdot (\ind \otimes \chi)}{\mX} = 0$.
\end{proof}

\begin{cor} \label{cor: Kronecker limit formula}
	Let $\bfZ = (Z, \alpha)$ be an ergodic Kronecker system,
	$H$ a compact abelian group, and $\sigma : \Gamma \times Z \to H$ a cocycle.
	Let $M$ be the Mackey group associated to $\sigma$, and let $\omega : Z \times Z \to H$
	be a measurable map such that $t \mapsto \omega(t, \cdot) + M$ is a continuous map $Z \to \M(Z, H/M)$
	and $\omega(\alpha_g, z) \equiv \sigma_g(z) \pmod{M}$.
	Then for any $f \in L^2(Z \times H)$,
	\begin{equation} \label{eq: Mackey limit formula}
		\UClim_{g \in \Gamma}{f(z + \alpha_g, x + \sigma_g(z))} = \int_{Z \times M}{f(z + t, x + m + \omega(t,z))~dt~dm}
	\end{equation}
	in $L^2(Z \times H)$.
\end{cor}

\begin{rem}
	Proposition \ref{prop: continuous extension mod Mackey} provides a function $\tilde{\omega} : Z \times Z \to H/M$
	corresponding to the cocycle $\sigma : \Gamma \times Z \to H$.
	This can be lifted to a measurable function $\omega : Z \times Z \to H$
	satisfying the conditions in the statement of Corollary \ref{cor: Kronecker limit formula}
	by the Kuratowski and Ryll-Nardewski measurable selection theorem (see \cite[Section 5.2]{srivastava}).
\end{rem}

\begin{proof}[Proof of Corollary \ref{cor: Kronecker limit formula}]
	By linearity, we may assume $f$ is of the form $f = h \otimes \chi$ with $\chi \in \hat{H}$.
	Then the right hand side of \eqref{eq: Mackey limit formula} is equal to
	\begin{equation} \label{eq: right hand side 1}
		\chi(x) \left( \int_Z{h(z + t) \chi(\omega(t,z))~dt} \right) \left( \int_M{\chi(m)~dm} \right).
	\end{equation}
	If $\chi \notin M^{\perp}$, then $\UClim_{g \in \Gamma}{T^{\sigma}_gf} = 0$ by Proposition \ref{prop: Mackey ET}.
	Moreover, \eqref{eq: right hand side 1} is clearly equal to zero.
	
	Suppose $\chi \in M^{\perp}$.
	Then \eqref{eq: right hand side 1} reduces to
	\begin{equation*}
		\chi(x) \left( \int_Z{h(z+t) \chi(\omega(t,z))~dt} \right).
	\end{equation*}
	The left hand side of \eqref{eq: Mackey limit formula} is equal to
	\begin{equation*}
		\chi(x) \cdot \UClim_{g \in \Gamma}{h(z + \alpha_g) \chi(\sigma_g(z))}.
	\end{equation*}
	Define $F : Z \times Z \to S^1$ by $F(t, z) := h(z + t) \chi(\omega(t,z))$.
	Then by Proposition \ref{prop: continuous extension mod Mackey},
	$t \mapsto F(t, \cdot)$ is a continuous map from $Z$ to $L^2(Z)$
	and $F(\alpha_g, z) = h(z + \alpha_g) \chi(\sigma_g(z))$.
	Since $(Z, \alpha)$ is uniquely ergodic, it follows by Lemma \ref{lem: Hil-valued ET} that
	\begin{equation*}
		\UClim_{g \in \Gamma}{F(\alpha_g, z)} = \int_Z{F(t, z)~dt}
	\end{equation*}
	in $L^2(Z)$.
	This completes the proof.
\end{proof}


\section{Extensions} \label{sec: extensions}

\begin{defn}
	Let $\Phi$ be a family of endomorphisms of $\Gamma$.
	A group of characters $\Lambda \subseteq \hat{\Gamma}$ is
	\begin{itemize}
		\item	\emph{$\Phi$-complete} if $\lambda \circ \varphi \in \Lambda$
			for every $\lambda \in \Lambda$ and every $\varphi \in \Phi$.
		\item	\emph{$\Phi$-divisible} if for every $\lambda \in \Lambda$ and every $\varphi \in \Phi$,
			there exists $\lambda' \in \Lambda$ such that $\lambda' \circ \varphi = \lambda$.
	\end{itemize}
\end{defn}

The main result of this section is the following extension theorem:

\begin{thm} \label{thm: complete divisible extension}
	Let $\Phi$ be any countable family of endomorphisms of $\Gamma$,
	and let $\Psi$ be a countable family of injective endomorphisms of $\Gamma$.
	Then for any ergodic $\Gamma$-system $\bfX = \left( X, \mX, \mu, (T_g)_{g \in \Gamma} \right)$,
	there exists an ergodic extension $\tilde{\bfX}$ of $\bfX$ such that the discrete spectrum of $\tilde{\bfX}$
	is $\Phi$-complete and $\Psi$-divisible.
\end{thm}

We will derive Theorem \ref{thm: complete divisible extension} as a consequence of the following general result:

\begin{thm} \label{thm: extension}
	Let $\bfX = \left( X, \mX, \mu, (T_g)_{g \in \Gamma} \right)$ be an ergodic $\Gamma$-system,
	and let $\Lambda$ be the discrete spectrum of $\bfX$.
	For any countable set $C \subseteq \hat{\Gamma}$, there is an ergodic extension $\tilde{\bfX}$ of $\bfX$
	such that the discrete spectrum $\tilde{\Lambda}$ of $\tilde{\bfX}$ is equal to the group generated by $\Lambda$ and $C$.
\end{thm}
\begin{proof}
	This is a special case $(k=1)$ of Theorem \ref{thm: HK extension}.
	Indeed, if $\tilde{\bfZ}$ is an ergodic Kronecker system with discrete spectrum
	$\tilde{\Lambda} = \left\langle \Lambda, C \right\rangle$, then $\tilde{\bfZ}$ is an order 1 extension of $\bfZ$,
	and so there exists by Theorem \ref{thm: HK extension} an ergodic extension $\tilde{\bfX}$ of $\bfX$
	whose Kronecker factor is isomorphic to $\tilde{\bfZ}$.
\end{proof}

\begin{rem}
	(1) A similar statement to Theorem \ref{thm: extension} is proved in \cite[Theorem 4.3]{abs}.
	However, in \cite[Theorem 4.3]{abs}, only one inclusion is established
	(namely, that $\tilde{\Lambda}$ contains $\Lambda$ and $C$).
	It is nevertheless true that the construction in the proof of \cite[Theorem 4.3]{abs}
	produces a system with the correct discrete spectrum, as can be seen with some additional work.
	(In fact, though differently-phrased, the construction is equivalent to taking the relatively independent joining
	of $\bfX$ with the appropriate ergodic Kronecker system $\tilde{\bfZ}$).
	
	(2) The extension $\tilde{\bfX}$ produced via Theorem \ref{thm: extension} is not unique.
	For instance, any weakly mixing extension of $\tilde{\bfX}$ will again satisfy the conclusion of the theorem.
	The minimal such extension (i.e., the one appearing as a factor of any such extension)
	is precisely the relatively independent joining $\bfX \times_{\bfZ} \tilde{\bfZ}$.
\end{rem}

Now we can prove Theorem \ref{thm: complete divisible extension}.

\begin{proof}[Proof of Theorem \ref{thm: complete divisible extension}]
	Let $\Lambda$ be the discrete spectrum of $\bfX$.
	We extend $\Lambda$ in stages, alternating between $\Phi$-completeness and $\Psi$-divisibility.
	
	First, we make a couple of convenient reductions to make the notation less cumbersome.
	Replacing $\Phi$ by $\Phi \cup \Psi$, we may assume that $\Psi \subseteq \Phi$.
	Let $\tilde{\Phi}$ be the semigroup
	\begin{align*}
		\tilde{\Phi} & = \left\{ \varphi_1 \circ \dots \circ \varphi_k : k \ge 0, \varphi_1, \dots, \varphi_k \in \Phi \right\},
	\intertext{and let $\tilde{\Psi}$ be the semigroup}
		\tilde{\Psi} & = \left\{ \psi_1 \circ \dots \circ \psi_k : k \ge 0, \psi_1, \dots, \psi_k \in \Psi \right\}.
	\end{align*}
	Since $\Phi$ and $\Psi$ are countable, $\tilde{\Phi}$ and $\tilde{\Psi}$ are also countable.
	Moreover, for $\tilde{\psi} = \psi_1 \circ \dots \circ \psi_k \in \tilde{\Psi}$, we have that $\tilde{\psi}$ is injective,
	since it is a composition of injective maps.
	Thus, replacing $\Phi$ and $\Psi$ with $\tilde{\Phi}$ and $\tilde{\Psi}$,
	we may assume without loss of generality that $\Phi$ and $\Psi$ contain the identity map
	and are closed under composition.
	
	Now we set up the induction process.
	Let $\Lambda_0 := \Lambda$.
	Suppose we have defined $\Lambda_0 \subseteq \dots \subseteq \Lambda_{2j}$ for some $j \ge 0$.
	Let
	\begin{equation*}
		\Lambda_{2j+1} := \left\langle \lambda \circ \varphi : \lambda \in \Lambda_{2j}, \varphi \in \Phi \right\rangle.
	\end{equation*}
	By the induction hypothesis, $\Lambda_{2j}$ is a countable group, so $\Lambda_{2j+1}$ is also a countable group.
	Moreover, since $\Phi$ is a semigroup, $\Lambda_{2j+1}$ is $\Phi$-complete.
	
	We now perform subinduction to define $\Lambda_{2j+2}$.
	Put $S_0 := \Lambda_{2j+1}$.
	Suppose we have defined $S_k$ for some $k \ge 0$.
	For $\lambda \in S_k$ and $\psi \in \Psi$, there exists $\lambda' \in \hat{\Gamma}$
	such that $\lambda' \circ \psi = \lambda$.
	To see this, first define $\lambda'_0 : \psi(\Gamma) \to S^1$ by $\lambda'_0(\psi(g)) = \lambda(g)$.
	This is well-defined since $\psi$ is injective.
	Then by \cite[Theorem 2.1.4]{rudin}, $\lambda'_0$ extends to a character $\lambda' \in \hat{\Gamma}$,
	and we have $\lambda' \circ \psi = \lambda$.
	Define a choice function $\gamma_k : S_k \times \Psi \to \hat{\Gamma}$
	so that $\gamma_k(\lambda, \psi) \circ \psi = \lambda$ for every $\lambda \in S_k$ and $\psi \in \Psi$.
	Then let
	\begin{equation*}
		S_{k+1} := \left\{ \gamma_k(\lambda, \psi) : \lambda \in S_k, \psi \in \Psi \right\}.
	\end{equation*}
	By induction, the set $S_{k+1}$ is countable.
	Therefore, $S = \bigcup_{k \ge 0}{S_k}$ is countable and hence generates a countable group
	$\Lambda_{2j+2} := \left\langle S \right\rangle$.
	
	We claim that $\Lambda_{2j+2}$ is $\Psi$-divisible.
	Let $\lambda \in \Lambda_{2j+2}$, and let $\psi \in \Psi$.
	We may write $\lambda = \prod_{i=1}^r{\gamma_{k_i}^{\eps_i}(\lambda_i, \psi_i)}$
	with $\lambda_i \in S_{k_i}$, $\psi_i \in \Psi$, and $\eps_i \in \{-1,1\}$.
	Let $\lambda'_i = \gamma_{k_i+1} \left( \gamma_{k_i}(\lambda_i, \psi_i), \psi \right) \in S$,
	and let $\lambda' = \prod_{i=1}^r{(\lambda'_i)^{\eps_i}} \in \Lambda_{2j+2}$.
	Then
	\begin{equation*}
		\lambda' \circ \psi = \prod_{i=1}^r{(\lambda'_i \circ \psi)^{\eps_i}} = \prod_{i=1}^r{\left( \gamma_{k_i}(\lambda_i, \psi_i) \right)^{\eps_i}} = \lambda.
	\end{equation*}
	Thus, $\Lambda_{2j+2}$ is $\Psi$-divisible as claimed.
	
	By induction, we have constructed an infinite sequence of countable groups
	$\Lambda_0 \subseteq \Lambda_1 \subseteq \Lambda_2 \subseteq \dots$ such that
	$\Lambda_{2j+1}$ is $\Phi$-complete and $\Lambda_{2j+2}$ is $\Psi$-divisible for $j \ge 0$.
	Let $\Lambda_{\infty} := \bigcup_{j=0}^{\infty}{\Lambda_j}$.
	Then $\Lambda_{\infty}$ is a countable group that is $\Phi$-complete and $\Psi$-divisible.
	
	Finally, we apply Theorem \ref{thm: extension} to obtain an ergodic extension $\tilde{\bfX}$ of $\bfX$
	such that the discrete spectrum $\tilde{\Lambda}$ of $\tilde{\bfX}$ is equal to $\Lambda_{\infty}$.
\end{proof}

\begin{rem}
	The subinduction and use of the choice functions $\gamma_k$ in the construction of the group $\Lambda_{2j+2}$
	is solely used to ensure that $\Lambda_{2j+2}$ is countable.
	A $\Psi$-divisible group can be defined more directly, namely
	\begin{equation*}
		D_{2j+2} := \left\langle \lambda \in \hat{\Gamma} : \lambda \circ \psi \in \Lambda_{2j+1}~\text{for some}~\psi \in \Psi \right\rangle.
	\end{equation*}
	In general, $D_{2j+2}$ is uncountable.
	However, if $\psi(\Gamma)$ is a finite index subgroup of $\Gamma$ for each $\psi \in \Psi$,
	then the set $\{\lambda \in \hat{\Gamma} : \lambda \circ \psi = \lambda_0\}$
	has cardinality $[\Gamma : \psi(\Gamma)] < \infty$ for each $\lambda_0 \in \Lambda_{2j+1}$ and $\psi \in \Psi$.
	Thus, in this case, $D_{2j+2}$ is countable, so one may take $\Lambda_{2j+2} = D_{2j+2}$
	rather than using the more complicated construction
	appearing in the proof of Theorem \ref{thm: complete divisible extension}.
	For the group $\Gamma = \Z^d$, an endomorphism is
	injective if and only if it has finite index image if and only if the corresponding matrix is nonsingular.
	Therefore, the simpler construction $\Lambda_{2j+2} = D_{2j+2}$
	can always be used when dealing with the group $\Gamma = \Z^d$.
\end{rem}


\section{Mackey group associated with $\{\varphi, \psi\}$} \label{sec: Mackey}

Let $\varphi, \psi \in \End(\Gamma)$ such that $(\psi - \varphi)(\Gamma)$ has finite index in $\Gamma$,
and suppose $\theta_1, \theta_2 \in \End(\Gamma)$ are such that $\theta_1 \circ \varphi + \theta_2 \circ \psi$ is injective.

Let $\bfX$ be an ergodic quasi-affine $\Gamma$-system, and write $\bfX = \bfZ \times_{\sigma} H$.
Assume that the discrete spectrum of $\bfX$ is $\{\varphi, \psi, \theta_1, \theta_2\}$-complete
and $(\theta_1 \circ \varphi + \theta_2 \circ \psi)$-divisible.
In this section, we consider a variant of the notion of Mackey groups as discussed in Section \ref{sec: Mackey prelim} that is tailored to analyzing ergodic averages of the form
\begin{equation*}
	\UClim_{g \in \Gamma} T_{\varphi(g)} f_1 \cdot T_{\psi(g)} f_2
\end{equation*}
for $f_1, f_2 \in L^{\infty}(Z \times H)$.

Since the discrete spectrum $\Lambda$ of $\bfX$ is $\{\varphi, \psi, \theta_1, \theta_2\}$-complete,
we have induced continuous endomorphisms of $Z = \hat{\Lambda}$,
which we denote by $\hat{\varphi}$, $\hat{\psi}$, $\hat{\theta}_1$, and $\hat{\theta}_2$.
To see this, view $Z$ as the dual group, expressed additively as the group of homomorphisms $z : \Lambda \to \T = \R/\Z$.
For $f \in \{\varphi, \psi, \theta_1, \theta_2\}$, the map $\hat{f} : Z \to Z$ is then given by
\begin{equation*}
	\hat{f}(z) : \lambda \mapsto z \left( \lambda \circ f \right).
\end{equation*}

Since $\Lambda$ is $(\theta_1 \circ \varphi + \theta_2 \circ \psi)$-divisible,
we claim that $\hat{\theta}_1 \circ \hat{\varphi} + \hat{\theta}_2 \circ \hat{\psi}$ is injective.
Indeed, suppose $(\hat{\theta}_1 \circ \hat{\varphi} + \hat{\theta}_2 \circ \hat{\psi})(z) = 0$.
Then for every $\lambda \in \Lambda$,
\begin{equation*}
	z \left( \lambda \circ (\theta_1 \circ \varphi + \theta_2 \circ \psi) \right) = 0.
\end{equation*}
But for any $\lambda \in \Lambda$, there exists $\lambda' \in \Lambda$
with $\lambda' \circ (\theta_1 \circ \varphi + \theta_2 \circ \psi) = \lambda$, so
\begin{equation*}
	z(\lambda) = z \left( \lambda' \circ (\theta_1 \circ \varphi + \theta_2 \circ \psi) \right) = 0.
\end{equation*}
That is, $z = 0$. \\

Let
\begin{equation*}
	W = \left\{ (z + \hat{\varphi}(t), z + \hat{\psi}(t)) : z, t \in Z \right\},
\end{equation*}
and let $\tilde{\alpha} : \Gamma \to W$ be the homomorphism
\begin{equation*}
	\tilde{\alpha}_g = \left( \alpha_{\varphi(g)}, \alpha_{\psi(g)} \right).
\end{equation*}
Let $\tilde{\sigma} : \Gamma \times W \to H^2$ be the cocycle
\begin{equation*}
	\tilde{\sigma}_g(w) = \left( \sigma_{\varphi(g)}(w_1), \sigma_{\psi(g)}(w_2) \right).
\end{equation*}

Now we define the Mackey group to be the closed subgroup $M \le H^2$ with annihilator
\begin{equation*}
	M^{\perp} = \left\{ \tilde{\chi} \in \hat{H^2} : \tilde{\chi} \circ \tilde{\sigma}~\text{is a coboundary over}~(W, \tilde{\alpha}) \right\}.
\end{equation*}

The Kronecker system $\bfW = (W, \tilde{\alpha})$ is not necessarily ergodic.
However, its ergodic decomposition is easy to describe and interacts well with the Mackey group $M$.
Namely, we may express $W$ as the union of the subsets
\begin{equation*}
	W_z = \left\{ (z + \hat{\varphi}(t), z + \hat{\psi}(t)) : t \in Z \right\} = (z,z) + W_0,
\end{equation*}
each supporting a Haar measure $m_z$.
The system $(W_0, \tilde{\alpha})$ is uniquely ergodic, and each of the systems $(W_z, \tilde{\alpha})$ for $z \in Z$
is an isomorphic copy.
It is easily verified that the Haar measure on $W$ decomposes as $m_W = \int_Z{m_z~dz}$.
An important property of this ergodic decomposition is that,
letting $M_z$ denote the Mackey group corresponding to the ergodic component $(W_z, \tilde{\alpha})$,
one has $M_z = M$ for a.e. $z \in Z$ (see \cite[Proposition 7.9]{abb}).

We now seek to describe the structure of the Mackey group $M$.
A classical fact in ergodic theory is that, given two measure-preserving systems $\bfX_1$ and $\bfX_2$,
invariant functions for the product system $\bfX_1 \times \bfX_2$ are formed from functions of the form $f_1 \otimes f_2$,
where $f_1$ is an eigenfunction of $\bfX_1$, $f_2$ is an eigenfunction of $\bfX_2$, and the corresponding eigenvalues
are conjugates of one another.
The following result describes the Mackey group $M$ in an analogous manner:

\begin{thm} \label{thm: Mackey gluing}
	Let $M$ be the Mackey group as defined above.
	Then
	\begin{equation*}
		M^{\perp} = \left\{ \chi_1 \otimes \chi_2 \in \hat{H^2} : \exists \gamma \in \hat{\Gamma}, \chi_1(\sigma_{\varphi(g)}(w_1)) \sim \gamma(g)~\text{and}~\chi_2(\sigma_{\psi(g)}(w_2)) \sim \overline{\gamma}(g)~\text{over}~(W, \tilde{\alpha}) \right\}.
	\end{equation*}
\end{thm}

\begin{rem}
	Throughout this section, we treat $g \in \Gamma$ and $w = (w_1, w_2) \in W$ as variables and write expressions of the form $\rho(g,w_1,w_2) \sim \tau(g,w_1,w_2)$ as shorthand for the statement that the cocycles $(g, (w_1,w_2)) \mapsto \rho(g, w_1, w_2)$ and $(g, (w_1, w_2)) \mapsto \tau(g, w_1, w_2)$ are cohomologous.
	For example, the notation $\chi_1(\sigma_{\varphi(g)}(w_1)) \sim \gamma(g)$ means that there is a measurable function $F : W \to S^1$ such that
	\begin{equation*}
		\chi_1(\sigma_{\varphi(g)}(w_1)) = \gamma(g) \frac{F\left( w_1 + \alpha_{\varphi(g)}, w_2 + \alpha_{\psi(g)} \right)}{F(w_1, w_2)}
	\end{equation*}
	for every $g \in \Gamma$ and almost every $w = (w_1, w_2) \in W$.
\end{rem}

\begin{proof}
	Suppose $\chi_1, \chi_2 \in \hat{H}$ and $\gamma \in \hat{\Gamma}$ such that
	$\chi_1(\sigma_{\varphi(g)}(w_1)) \sim \gamma(g)$ and $\chi_2(\sigma_{\psi(g)}(w_2)) \sim \overline{\gamma}(g)$.
	Then
	\begin{equation*}
		\left( \chi_1 \otimes \chi_2 \right)(\tilde{\sigma}_g(w)) = \chi_1(\sigma_{\varphi(g)}(w_1)) \chi_2(\sigma_{\psi(g)}(w_2)) \sim \gamma(g) \overline{\gamma}(g) = 1
	\end{equation*}
	so $\chi_1 \otimes \chi_2 \in M^{\perp}$. \\
	
	Conversely, suppose $\chi_1 \otimes \chi_2 \in M^{\perp}$.
	Let $(g_n)_{n \in \N}$ be a sequence in $\Gamma$ such that $\tilde{\alpha}_{g_n} \to 0$ in $W$.
	By Lemma \ref{lem: coboundary},
	\begin{equation*}
		(\chi_1 \otimes \chi_2) \circ \tilde{\sigma}_{g_n} \to 1
	\end{equation*}
	in $L^2(W)$.
	That is,
	\begin{equation} \label{eq: Z x Z conv}
		\chi_1 \left( \sigma_{\varphi(g_n)} \left( z + \hat{\varphi}(t) \right) \right) \chi_2 \left( \sigma_{\psi(g_n)} \left( z + \hat{\psi}(t) \right) \right) \to 1
	\end{equation}
	in $L^2(Z \times Z)$.
	
	The cocycle $\sigma$ is quasi-affine, so by Lemma \ref{lem: quasi-affine},
	there are sequences $(c_{i,n})_{n \in \N}$ in $S^1$ and $(\lambda_{i,n})_{n \in \N}$ in $\hat{Z}$ for $i = 1, 2$ such that
	\begin{equation} \label{eq: QA convergence}
		c_{1,n} \lambda_{1,n}(z) \chi_1 \left( \sigma_{\varphi(g_n)}(z) \right) \to 1 \qquad \text{and} \qquad
		c_{2,n} \lambda_{2,n}(z) \chi_2 \left( \sigma_{\psi(g_n)}(z) \right) \to 1
	\end{equation}
	in $L^2(Z)$.
	We will combine \eqref{eq: Z x Z conv} and \eqref{eq: QA convergence} to show that
	$\chi_1(\sigma_{\varphi(g)}(z)) \sim \gamma(g)$ and $\chi_2(\sigma_{\psi(g)}(z)) \sim \overline{\gamma}(g)$
	for some $\gamma \in \hat{\Gamma}$.
	
	For convenience, let $\mu_n = \chi_1 \circ \sigma_{\varphi(g_n)}$ and $\nu_n = \chi_2 \circ \sigma_{\psi(g_n)}$.
	Now we perform a change of coordinates.
	Define $\eta : Z^2 \to Z^2$ by
	\begin{equation*}
		\eta(z,t) = \left( z + \hat{\varphi}(t), z + \hat{\psi}(t) \right)
	\end{equation*}
	and, for $u \in Z$, let $\zeta_u : Z^2 \to Z^2$ be the map
	\begin{equation*}
		\zeta_u(z,t) = \left( u + \hat{\psi}(z) - \hat{\varphi}(t), t - z \right).
	\end{equation*}
	Note that \eqref{eq: Z x Z conv} is equivalent to
	\begin{equation} \label{eq: Z^2 conv equiv}
		(\mu_n \otimes \nu_n) \circ \eta \to 1
	\end{equation}
	in $L^2 \left( Z^2 \right)$.
	
	We claim
	\begin{equation} \label{eq: conv claim}
		(\mu_n \otimes \nu_n) \circ \eta \circ \zeta_u \to 1
	\end{equation}
	in $L^2 \left( Z^2 \right)$.
	Let $\zeta := \zeta_0$.
	Fix $u \in Z$, and let $f_n(z,t) = \left( \left( \mu_n \otimes \nu_n \right) \circ \eta \right) (z+u,t)$.
	Then $\left( \mu_n \otimes \nu_n \right) \circ \eta \circ \zeta_u = f_n \circ \zeta$.
	We then want to show $f_n \circ \zeta \to 1$ in $L^2 \left( Z^2 \right)$.
	Since the Haar measure on $Z^2$ is invariant under shifting by $(u,0)$, we have $f_n \to 1$ in $L^2 \left( Z^2 \right)$
	by \eqref{eq: Z^2 conv equiv}.
	It therefore suffices to show that $\zeta \left( Z^2 \right)$ has positive measure (equivalently, finite index) in $Z^2$.
	By assumption, $(\psi - \varphi)(\Gamma)$ has finite index in $\Gamma$.
	Hence, $\left[ Z : \left( \hat{\psi} - \hat{\varphi} \right)(Z) \right] \le \left[ \Gamma : (\psi - \varphi)(\Gamma) \right] < \infty$.
	Let $F \subseteq Z$ be a finite set such that $\left( \hat{\psi} - \hat{\varphi} \right)(Z) + F = Z$.
	Let $(z,t) \in Z^2$ be given.
	Choose $x \in Z$ and $s \in F$ such that $\left( \hat{\psi} - \hat{\varphi} \right)(x) + s = z + \hat{\varphi}(t)$.
	Put $y = x + t$.
	Then
	\begin{equation*}
		\zeta(x,y) + (s,0) = \left( \hat{\psi}(x) - \hat{\varphi}(x) - \hat{\varphi}(t) + s, t \right) = (z,t).
	\end{equation*}
	This shows that $\zeta \left( Z^2 \right) + \left( F \times \{0\} \right) = Z^2$, so
	\begin{equation*}
		m_{Z^2} \left( \zeta \left( Z^2 \right) \right) \ge \frac{1}{|F|} > 0.
	\end{equation*}
	Thus, \eqref{eq: conv claim} holds.
	
	But
	\begin{equation*}
		\left( \eta \circ \zeta_u \right) (z,t) = \eta \left( u + \hat{\psi}(z) - \hat{\varphi}(t), t - z \right) = \left( u + \left( \hat{\psi} - \hat{\varphi} \right)(z), u + \left( \hat{\psi} - \hat{\varphi} \right)(t) \right).
	\end{equation*}
	We therefore deduce from \eqref{eq: conv claim} that
	\begin{equation} \label{eq: conv on u coset}
		\mu_n \otimes \nu_n \to 1
	\end{equation}
	in $L^2 \left( \left( u + \left( \hat{\psi} - \hat{\varphi} \right)(Z) \right)^2 \right)$.
	Taking a conjugate and multiplying by \eqref{eq: QA convergence}, we deduce
	\begin{equation*}
		c_{1,n} c_{2,n} (\lambda_{1,n} \otimes \lambda_{2,n}) \to 1
	\end{equation*}
	in $L^2 \left( \left( u + \left( \hat{\psi} - \hat{\varphi} \right)(Z) \right)^2 \right)$.
	That is,
	\begin{equation*}
		\int_{Z^2}{\left| c_{1,n} c_{2,n} \lambda_{1,n} \left( u + \left( \hat{\psi} - \hat{\varphi} \right)(z) \right) \lambda_{2,n} \left( u + \left( \hat{\psi} - \hat{\varphi} \right)(t) \right) - 1\right|^2~dz~dt} \to 0.
	\end{equation*}
	Multiplying by $\overline{\lambda}_{2,n}$ in the integrand and using the fact that each $\lambda_{i,n}$ is a homomorphism,
	\begin{equation*}
		\int_{Z^2}{\left| c_{1,n} c_{2,n} \lambda_{1,n}(u) \lambda_{1,n} \left( \left( \hat{\psi} - \hat{\varphi} \right)(z) \right) - \overline{\lambda}_{2,n}(u) \overline{\lambda}_{2,n} \left( \left( \hat{\psi} - \hat{\varphi} \right)(t) \right) \right|^2~dz~dt} \to 0.
	\end{equation*}
	It follows that for all sufficiently large $n$, we have
	$\lambda_{1,n}, \lambda_{2,n} \in \left( \left( \hat{\psi} - \hat{\varphi} \right)(Z) \right)^{\perp}$.
	
	Let $A = \left( \hat{\psi} - \hat{\varphi} \right)(Z)$, and put $A_1 = \left\{ (w_1, w_2) \in W : w_1 \in A \right\}$.
	Note that $A_1$ is a Borel set, and $m_W(A_1) = m_Z(A) > 0$.
	Fix $t = (t_1, t_2) \in A_1$ and let
	\begin{equation*}
		\rho_g(w) = \frac{\chi_1\left( \sigma_{\varphi(g)}(w_1 + t_1) \right)}{\chi_1 \left( \sigma_{\varphi(g)}(w_1) \right)}.
	\end{equation*}
	By Lemma \ref{lem: cohomologous to character}, our goal is to show that $\rho$ is a coboundary.
	Since $(g_n)_{n \in \N}$ is an arbitrary sequence in $\Gamma$ with $\tilde{\alpha}_{g_n} \to 0$,
	it suffices by Lemma \ref{lem: coboundary} to show $\rho_{g_n} \to 1$ in $L^2(W)$.
	We have already seen that
	\begin{equation*}
		c_{1,n} \lambda_{1,n}(z) \chi_1\left( \sigma_{\varphi(g_n)}(z) \right) \to 1
	\end{equation*}
	in $L^2(Z)$ and $\lambda_{1,n} \in A^{\perp}$.
	In particular, $\lambda_{1,n}(z+t_1) = \lambda_{1,n}(z)$, since $t_1 \in A$.
	Therefore,
	\begin{equation*}
		\rho_{g_n}(w) = \frac{\chi_1\left( \sigma_{\varphi(g_n)}(w_1 + t_1) \right)}{\chi_1 \left( \sigma_{\varphi(g_n)}(w_1) \right)} = \frac{c_{1,n} \lambda_{1,n}(w_1 + t_1) \chi_1\left( \sigma_{\varphi(g_n)}(w_1 + t_1) \right)}{c_{1,n} \lambda_{1,n}(w_1) \chi_1 \left( \sigma_{\varphi(g_n)}(w_1) \right)} \to 1
	\end{equation*}
	in $L^2(W)$.
	This proves that $\chi_1 \left( \sigma_{\varphi(g)}(w_1) \right)$ is cohomologous to a character.
	The same argument applies to $\chi_2 \left( \sigma_{\psi(g)}(w_2) \right)$,
	taking the set $A_2 = \left\{ (w_1, w_2) \in W : w_2 \in A \right\}$.
	
	Therefore, we may assume in the above that $\lambda_{i,n} = 1$ for every $n \in \N$ and $i = 1, 2$.
	We may also assume $c_{i,n} = \gamma_i(g_n)$, where $\chi_1 \left( \sigma_{\varphi(g)}(w_1) \right) \sim \gamma_1(g)$
	and $\chi_2 \left( \sigma_{\psi(g)}(w_2) \right) \sim \gamma_2(g)$.
	We have thus shown that for any sequence $(g_n)_{n \in \N}$ with $\tilde{\alpha}_{g_n} \to 0$, we have
	\begin{equation*}
		\left( \gamma_1 \gamma_2 \right)(g_n) \to 1.
	\end{equation*}
	Hence, by Lemma \ref{lem: coboundary}, $\gamma_1 \gamma_2 \sim 1$.
	By tweaking $\gamma_2$ up to cohomology, we may assume $\gamma_2 = \overline{\gamma}_1$, completing the proof.
\end{proof}

By passing to an extension of the original system, we will show that the Mackey group $M$ associated with $\{\varphi, \psi\}$
decomposes into the Cartesian product of Mackey groups associated with $\varphi$ and $\psi$ respectively:

\begin{thm} \label{thm: Mackey prod}
	Let $\varphi, \psi \in \End(\Gamma)$ such that $(\psi - \varphi)(\Gamma)$ has finite index in $\Gamma$,
	and suppose $\theta_1, \theta_2 \in \End(\Gamma)$
	such that $\theta_1 \circ \varphi + \theta_2 \circ \psi$ is injective.
	Let $\bfX = \bfZ \times_{\sigma} H$ be an ergodic quasi-affine $\Gamma$ system
	such that the discrete spectrum $\hat{Z}$ is $\{\varphi, \psi, \theta_1, \theta_2\}$-complete
	and $(\theta_1 \circ \varphi + \theta_2 \circ \psi)$-divisible.
	There is an ergodic quasi-affine extension $\bfX' = \bfZ' \times_{\sigma'} H$ of $\bfX$
	such that the Mackey group $M'$ decomposes as $M' = M'_{\varphi} \times M'_{\psi}$, where
	\begin{align*}
		\left( M'_{\varphi} \right)^{\perp} & = \left\{ \chi \in \hat{H} : \chi \left( \sigma'_{\varphi(g)}(w_1) \right)
		 ~\text{is a coboundary over}~\left( W', \tilde{\alpha}' \right) \right\}
	\intertext{and}
		\left( M'_{\psi} \right)^{\perp} & = \left\{ \chi \in \hat{H} : \chi \left( \sigma'_{\psi(g)}(w_2) \right)
		 ~\text{is a coboundary over}~\left( W', \tilde{\alpha}' \right) \right\}.
	\end{align*}
\end{thm}
\begin{proof}
	Let
	\begin{align*}
		C_{\varphi} & = \left\{ \gamma \in \hat{\Gamma} : \exists \chi \in \hat{H}, \chi(\sigma_{\varphi(g)}(w_1)) \sim \gamma(g)~\text{over}~(W, \tilde{\alpha}) \right\}
	\intertext{and}
		C_{\psi} & = \left\{ \gamma \in \hat{\Gamma} : \exists \chi \in \hat{H}, \chi(\sigma_{\psi(g)}(w_2)) \sim \gamma(g)~\text{over}~(W, \tilde{\alpha}) \right\}.
	\end{align*}
	Then let $C = C_{\varphi} \cap C_{\psi}$.
	Note that a character $\gamma \in \hat{\Gamma}$ is cohomologous to another character $\gamma' \in \hat{\Gamma}$
	if and only if $\overline{\gamma} \gamma'$ is an eigenvalue for the system $(W, \tilde{\alpha})$.
	That is,
	\begin{equation*}
		\overline{\gamma} \gamma' \in \left\{ (\lambda_1 \circ \varphi)(\lambda_2 \circ \psi) : \lambda_1, \lambda_2 \in \hat{Z} \right\} \subseteq \hat{Z}.
	\end{equation*}
	Moreover, the groups $\hat{Z}$ and $\hat{H}$ are countable, so $C$ is a countable set of characters.
	
	By Theorem \ref{thm: complete divisible extension}, let $\Lambda'$ be a countable subgroup of $\hat{\Gamma}$
	that is $\{\varphi, \psi, \theta_1, \theta_2\}$-complete, $(\theta_1 \circ \varphi + \theta_2 \circ \psi)$-divisible,
	and contains the group generated by $\Lambda = \hat{Z}$ and $C$, and let $Z' = \hat{\Lambda}'$.
	For $g \in \Gamma$, let $\alpha'_g \in Z'$ be the element such that $\alpha'_g(\lambda) = \lambda(g)$
	for every $\lambda \in \Lambda'$.
	Since $\Lambda \subseteq \Lambda'$,
	there is a surjective homomorphism $\pi : Z' \to Z$ such that $\pi(\alpha'_g) = \alpha_g$.
	Define $\sigma' : \Gamma \times Z' \to H$ by $\sigma'_g(z) = \sigma_g(\pi(z))$. \\
	
	\underline{Claim 1}: $\sigma'$ is a cocycle. \\
	
	Indeed, for any $g, h \in \Gamma$ and $z \in Z'$,
	\begin{align*}
		\sigma'_{g+h}(z) & = \sigma_{g+h}(\pi(z)) & (\text{definition of}~\sigma') \\
		 & = \sigma_g(\pi(z) + \alpha_h) + \sigma_h(\pi(z)) & (\sigma~\text{is a cocycle}) \\
		 & = \sigma_g(\pi(z + \alpha'_h)) + \sigma_h(\pi(z)) & (\pi~\text{is a factor map}) \\
		 & = \sigma'_g(z + \alpha'_h) + \sigma'_h(z).
	\end{align*}
	This proves the claim. \\
	
	\underline{Claim 2}: $\sigma'$ is quasi-affine. \\
	
	The cocycle $\sigma$ is quasi-affine, so there exist measurable functions $F : Z \times Z \to S^1$
	and $\gamma : Z \to \hat{\Gamma}$ such that
	\begin{equation*}
		\frac{\sigma_g(z + t)}{\sigma_g(z)} = \gamma(t,g) \frac{F(t,z + \alpha_g)}{F(t,z)}.
	\end{equation*}
	Define $F' : Z' \times Z' \to S^1$ by $F'(t,z) = F(\pi(t), \pi(z))$ and $\gamma' : Z' \to \hat{\Gamma}$ by $\gamma'(t, \cdot) = \gamma(\pi(t), \cdot)$.
	Then
	\begin{equation*}
		\frac{\sigma'_g(z + t)}{\sigma'_g(z)} = \frac{\sigma_g(\pi(z) + \pi(t))}{\sigma_g(\pi(z))} = \gamma(\pi(t),g) \frac{F(\pi(t), \pi(z) + \alpha_g)}{F(\pi(t), \pi(z))} = \gamma'(t,g) \frac{F'(t, z+\alpha'_g)}{F'(z)},
	\end{equation*}
	so $\sigma'$ is quasi-affine as claimed. \\
	
	\underline{Claim 3}: The quasi-affine system $\bfZ' \times_{\sigma'} H$ is an ergodic system
	with Kronecker factor $\bfZ' = (Z', \alpha')$. \\
	
	We need to check that $\sigma'$ is a weakly mixing cocycle.
	That is, for any $\chi \in \hat{H}$, if $\chi \circ \sigma'$ is cohomologous to a character, then $\chi = 1$
	(see \cite[Proposition 7.5]{abb}).
	Suppose $\chi \in \hat{H}$ and $\chi \circ \sigma' \sim \gamma \in \hat{\Gamma}$ over $\bfZ'$.
	Let $F : Z' \to S^1$ such that
	\begin{equation*}
		\chi(\sigma'_g(z)) = \gamma(g) \frac{F(z + \alpha'_g)}{F(z)}
	\end{equation*}
	for every $g \in \Gamma$ and almost every $z \in Z'$.
	Define $G : Z \times H \to S^1$ by $G(z,x) = \overline{F}(z) \chi(x)$.
	Then
	\begin{equation*}
		G \left( z + \alpha'_g, x + \sigma'_g(z) \right) = \overline{F} \left( z + \alpha'_g \right) \chi \left( \sigma'_g(z) \right) \chi(x) = \gamma(g) \overline{F}(z) \chi(x) = \gamma(g) G(z,x).
	\end{equation*}
	Hence, $G$ is an eigenfunction for the system $\bfZ' \times_{\sigma'} H$ (with eigenvalue $\gamma$).
	The function $F \otimes \ind_H$ is measurable with respect to the Kronecker factor of $\bfZ' \times_{\sigma'} H$,
	since $\bfZ'$ is a Kronecker system and therefore contained in the Kronecker factor.
	Therefore, $\ind_{Z'} \otimes \chi = (F \otimes \ind_H)G$ is measurable with respect to the Kronecker factor.
	
	Now, the projection of $\ind_{Z'} \otimes \chi$ under the factor map $\bfZ' \times_{\sigma'} H \to \bfZ \times_{\sigma} H$
	is the function $\ind_Z \otimes \chi$.
	Thus, $\ind_Z \otimes \chi$ is measurable with respect to the Kronecker factor of $\bfZ \times_{\sigma} H$,
	but by assumption, the Kronecker factor of $\bfZ \times_{\sigma} H$ is $\bfZ$.
	It follows that the function $(z,x) \mapsto \chi(x) = (\ind_Z \otimes \chi)(z,x)$ does not depend on $x$.
	That is, $\chi = 1$. \\
	
	As a brief aside, it is worth remarking that the system $\bfZ' \times_{\sigma'} H$ is (isomorphic to)
	the relatively independent joining of $\bfZ \times_{\sigma} H$ and $\bfZ'$
	with respect to the common factor $\bfZ$.
	We therefore could have shown the previous three claims by establishing this isomorphism
	and referring to Theorem \ref{thm: HK extension}.
	However, in order to prove that $\bfX' = \bfZ' \times_{\sigma'} H$ is the desired extension of $\bfX$,
	it is more convenient to work with the system written explicitly as a group extension over its Kronecker factor
	rather than appealing to general abstract statements about Host--Kra factors. \\
	
	It remains to show that the ergodic quasi-affine system $\bfZ' \times_{\sigma'} H$ is the desired extension.
	Let us introduce some notation.
	We define a system $(W', \tilde{\alpha}')$ by
	\begin{equation*}
		W' = \left\{ \left( z + \varphi(t), z + \psi(t) \right) : z, t \in Z' \right\}
	\end{equation*}
	and
	\begin{equation*}
		\tilde{\alpha}'_g = \left( \alpha'_{\varphi(g)}, \alpha'_{\psi(g)} \right).
	\end{equation*}
	We then define the cocycle $\tilde{\sigma}' : \Gamma \times W' \to H^2$ by
	\begin{equation*}
		\tilde{\sigma}'_g(w) = \left( \sigma'_{\varphi(g)}(w_1), \sigma'_{\psi(g)}(w_2) \right)
	\end{equation*}
	and associate a Mackey group $M'$ with annihilator
	\begin{equation*}
		(M')^{\perp} = \left\{ \tilde{\chi} \in \hat{H^2} : \tilde{\chi} \circ \tilde{\sigma}'~\text{is a coboundary over}~(W', \tilde{\alpha}') \right\}.
	\end{equation*}
	We want to show $M' = M'_{\varphi} \times M'_{\psi}$.
	To this end, we prove one more claim: \\
	
	\underline{Claim 4}: Let $\chi \in \hat{H}$ and $\gamma \in \hat{\Gamma}$.
	\begin{enumerate}[(a)]
		\item	If $\chi \left( \sigma'_{\varphi(g)}(w_1) \right) \sim \gamma(g)$ over $(W', \tilde{\alpha}')$,
			then there exists $\gamma' \in C_{\varphi}$ such that $\gamma \sim \gamma'$.
		\item	If $\chi \left( \sigma'_{\psi(g)}(w_2) \right) \sim \gamma(g)$ over $(W', \tilde{\alpha}')$,
			then there exists $\gamma' \in C_{\psi}$ such that $\gamma \sim \gamma'$.
	\end{enumerate}
	
	The proofs of items (a) and (b) are the same, so we prove only (a).
	Suppose $\chi \left( \sigma'_{\varphi(g)}(w_1) \right) \sim \gamma(g)$ over $\bfW' = (W', \tilde{\alpha}')$.
	Arguing as in the proof of Claim 3 above, the function $\ind_{W'} \otimes \chi \otimes \ind_H : W' \times H^2 \to S^1$
	is measurable with respect to the Kronecker factor of $\bfW' \times_{\tilde{\sigma}'} H^2$.
	Projecting onto the factor $\bfW \times_{\tilde{\sigma}} H^2$, it follows that $\ind_W \otimes \chi \otimes \ind_H$
	is measurable with respect to the Kronecker factor of $\bfW \times_{\tilde{\sigma}} H^2$.
	We then project again to the factor $(Z, \alpha \circ \varphi) \times_{\sigma_{\varphi}} H$ 
	(with the action of $\Gamma$ given by $g \cdot (z,x) = \left( z + \alpha_{\varphi(g)}, x + \sigma_{\varphi(g)}(z) \right)$)
	to conclude that $\ind_Z \otimes \chi$ is measurable with respect to the Kronecker factor $\bfZ_{\varphi}$
	of $(Z, \alpha \circ \varphi) \times_{\sigma_{\varphi}} H$.
	Noting that $L^2(\mZ_{\varphi})$ is spanned by functions of the form $F(z) \zeta(x)$
	with $\zeta \in \hat{H}$ such that $\zeta \circ \sigma_{\varphi}$ is cohomologous to a character
	(see the proof of \cite[Proposition 7.5(2)]{abb}), it follows that $\chi \circ \sigma_{\varphi}$ is cohomologous
	to a character $\gamma' \in \hat{\Gamma}$ over $(Z, \alpha \circ \varphi)$.
	Therefore, $\gamma' \in C_{\varphi}$ and $\gamma \sim \gamma'$. \\
	
	Let $\chi_1 \otimes \chi_2 \in (M')^{\perp}$.
	By Theorem \ref{thm: Mackey gluing}, there exists $\gamma \in \hat{\Gamma}$ such that
	$\chi_1(\sigma'_{\varphi(g)}(w_1)) \sim \gamma(g)$ and
	$\chi_2(\sigma'_{\psi(g)}(w_2)) \sim \overline{\gamma}(g)$ over $(W', \tilde{\alpha}')$.
	By Claim 4, we may assume $\gamma \in C_{\varphi} \cap C_{\psi} = C$.
	We constructed the extension $\bfZ' \times_{\sigma'} H$ so that $\Lambda' \supseteq C$.
	Therefore, $\gamma \in \Lambda'$.
	Also by construction, $\Lambda'$ is $\{\varphi, \psi, \theta_1, \theta_2\}$-complete
	and ($\theta_1 \circ \varphi + \theta_2 \circ \psi)$-divisible.
	Using the divisibility condition, let $\lambda \in \Lambda'$
	such that $\lambda \circ (\theta_1 \circ \varphi + \theta_2 \circ \psi) = \gamma$.
	Let $\lambda_1 = \lambda \circ \theta_1$ and $\lambda_2 = \lambda \circ \theta_2$.
	By the completeness condition, $\lambda_1, \lambda_2 \in \Lambda'$.
	Moreover, $(\lambda_1 \circ \varphi)(\lambda_2 \circ \psi) = \gamma$.
	So, taking $F = \left. \left( \lambda_1 \otimes \lambda_2 \right) \right|_{W'} : W' \to S^1$,
	we have $\gamma(g) = \Delta_gF$, so $\gamma \sim 1$.
	Thus,
	\begin{equation*}
		\chi_1 \left( \sigma'_{\varphi(g)}(w_1) \right) \sim \chi_2 \left( \sigma'_{\psi(g)}(w_2) \right) \sim 1.
	\end{equation*}
	That is, $\chi_1 \otimes \chi_2 \in \left( M'_{\varphi} \right)^{\perp} \times \left( M'_{\psi} \right)^{\perp}$ as desired.
\end{proof}


\section{Limit formula} \label{sec: limit}

In this section, we use the Mackey group $M$ defined in the previous section in order to derive a limit formula
for double ergodic averages over quasi-affine systems:

\begin{thm} \label{thm: limit}
	Let $\varphi, \psi, \theta_1, \theta_2 \in \End(\Gamma)$ such that $\psi - \varphi$ has finite index image in $\Gamma$
	and $\theta_1 \circ \varphi + \theta_2 \circ \psi$ is injective.
	Let $\bfX = \bfZ \times_{\sigma} H$ be an ergodic quasi-affine $\Gamma$-system
	whose discrete spectrum is $\{\varphi, \psi, \theta_1, \theta_2\}$-complete
	and $(\theta_1 \circ \varphi + \theta_2 \circ \psi)$-divisible.
	Let $M \le H^2$ be the Mackey group associated with $\{\varphi, \psi\}$.
	Then there is a measurable function $\omega : Z \times Z \to H^2$ such that
	\begin{enumerate}[(1)]
		\item	$\omega(0,z) \in M$ for all $z \in Z$
			and $t \mapsto \omega(t,\cdot) + M$ is a continuous function from $Z$ to $\M(Z, H^2/M)$; and
		\item	for any $f_1, f_2 \in L^{\infty}(Z \times H)$, we have
			\begin{align} \label{eq: limit}
				\UClim_{g \in \Gamma}&~{f_1(T_{\varphi(g)}(z,x)) f_2(T_{\psi(g)}(z,x))} \nonumber \\
				 & = \int_{Z \times M}
				 {f_1 \left( z + \hat{\varphi}(t), x + u + \omega_1(t,z) \right)
				 f_2 \left( z + \hat{\psi}(t), x + v + \omega_2(t,z) \right)~dm_Z(t)~dm_M(u,v)}
			\end{align}
			in $L^2(Z \times H)$, where $\omega = (\omega_1, \omega_2)$.
	\end{enumerate}
\end{thm}

Before proving Theorem \ref{thm: limit}, we note an immediate corollary:

\begin{cor} \label{cor: twisted limit}
	In the setup of Theorem \ref{thm: limit}, for any $f_0, f_1, f_2 \in L^{\infty}(Z \times H)$
	and any continuous function $\kappa : Z \to \C$, one has
	\begin{equation} \begin{split} \label{eq: twisted limit}
		\UClim_{g \in \Gamma}&~{\kappa(\alpha_g)
		 \int_{Z \times H}{f_0 \cdot T_{\varphi(g)} f_1 \cdot T_{\psi(g)} f_2~d(m_Z \times m_H)}} \\
		 = &~\int_{Z^2 \times H \times M}{\kappa(t)~f_0(z,x)~f_1 \left( z + \hat{\varphi}(t), x + u + \omega_1(t,z) \right)} \\
		 & \qquad \qquad {~f_2 \left( z + \hat{\psi}(t), x + v + \omega_2(t,z) \right)~dm_Z(z)~dm_Z(t)~dm_H(x)~dm_M(u,v)}.
	\end{split} \end{equation}
\end{cor}
\begin{proof}
	By the Stone--Weierstrass theorem, we may assume $\kappa$ is a character on $Z$.
	Since the discrete spectrum $\Lambda \cong \hat{Z}$ is $(\theta_1 \circ \varphi + \theta_2 \circ \psi)$-divisible,
	there exists $\kappa' \in \hat{Z}$ such that
	$\kappa' \circ \left( \hat{\theta}_1 \circ \hat{\varphi} + \hat{\theta}_2 \circ \hat{\psi} \right) = \kappa$.
	Define functions $\tilde{f}_i \in L^{\infty}(Z \times H)$ by
	\begin{align*}
		\tilde{f}_0(z,x) & = \overline{\kappa' \left( \hat{\theta}_1(z) + \hat{\theta}_2(z) \right)} f_0(z,x), \\
		\tilde{f}_1(z,x) & = \kappa' \left( \hat{\theta}_1(z) \right) f_1(z,x), \\
		\tilde{f}_2(z,x) & = \kappa' \left( \hat{\theta}_2(z) \right) f_2(z,x).
	\end{align*}
	
	The left hand side of \eqref{eq: twisted limit} is equal to
	\begin{equation} \label{eq: LHS}
		\UClim_{g \in \Gamma}{\int_{Z \times H}{\tilde{f}_0 \cdot T_{\varphi(g)} \tilde{f}_1 \cdot T_{\psi(g)} \tilde{f_2}}},
	\end{equation}
	while the right hand side of \eqref{eq: twisted limit} is equal to
	\begin{equation} \begin{split} \label{eq: RHS}
		\int_{Z^2 \times H \times M}&{\tilde{f}_0(z,x)~\tilde{f}_1 \left( z + \hat{\varphi}(t), x + u + \omega_1(t,z) \right)} \\
		 &~{\tilde{f}_2 \left( z + \hat{\psi}(t), x + v + \omega_2(t,z) \right)~dm_Z(z)~dm_Z(t)~dm_H(x)~dm_M(u,v)}.
	\end{split} \end{equation}
	The quantities in \eqref{eq: LHS} and \eqref{eq: RHS} are equal by Theorem \ref{thm: limit}.
\end{proof}

The first step in the proof of Theorem \ref{thm: limit} is the following enhancement of Proposition \ref{prop: Mackey ET}:

\begin{prop} \label{prop: Mackey orthogonality}
	Let $f_1, f_2 \in L^{\infty}(Z \times H)$.
	Suppose that for every $\tilde{\chi} \in M^{\perp}$, one has
	\begin{equation} \label{eq: orthogonal to chi}
		\int_{H^2}{f_1(w_1, x_1) f_2(w_2, x_2) \tilde{\chi}(x)~dx} = 0
	\end{equation}
	for a.e. $w = (w_1, w_2) \in W$.
	Then
	\begin{equation*}
		\UClim_{g \in \Gamma}{f_1(z + \alpha_{\varphi(g)}, x + \sigma_{\varphi(g)}(z)) f_2(z + \alpha_{\psi(g)}, x + \sigma_{\psi(g)}(z))} = 0
	\end{equation*}
	in $L^2(Z \times H)$.
\end{prop}
\begin{proof}
	The proof is very much in the spirit of \cite[Proposition 7.10]{abb}.
	Define $\tilde{T}_g : W \times H^2 \to W \times H^2$ by
	\begin{equation*}
		\tilde{T}_g(w, x) = \left( w + \tilde{\alpha}_g, x + \tilde{\sigma}_g(w) \right).
	\end{equation*}
	Set $F(w, x) := f_1(w_1, x_1) f_2(w_2, x_2)$ for $w = (w_1, w_2) \in W$ and $x = (x_1, x_2) \in H^2$.
	We claim $F$ is orthogonal to the space of $(\tilde{T}_g)_{g \in \Gamma}$-invariant functions in $L^2(W \times H^2)$.
	
	Let $E_1 := \{z \in Z : M_z = M\}$.
	As discussed in Section \ref{sec: Mackey}, $E_1$ has full measure.
	Let $E_2 := \{z \in Z : \eqref{eq: orthogonal to chi}~\text{holds for a.e.}~w = (w_1, w_2) \in W_z\}$,
	and let $E_3 := \{z \in Z : F \in L^{\infty} \left( W_z \times H^2 \right)\}$.
	By Fubini's theorem, both of the sets $E_2$ and $E_3$ have full measure in $Z$.
	Put $E := E_1 \cap E_2 \cap E_3$.
	
	Suppose $z \in E$, and let $\tilde{\chi} \in M_z^{\perp} = M^{\perp}$.
	Then
	\begin{equation*}
		\int_{H^2}{F(w, x) \tilde{\chi}(x)~dx} = 0
	\end{equation*}
	for a.e. $w \in W_z$.
	Therefore,
	\begin{equation*}
		\UClim_{g \in \Gamma}{\tilde{T}_gF} = 0
	\end{equation*}
	in $L^2 \left( W_z \times H^2 \right)$ by Proposition \ref{prop: Mackey ET}.
	
	Fix a F{\o}lner sequence $(\Phi_N)_{N \in \N}$ in $\Gamma$, and define the average
	\begin{equation*}
		A_N := \frac{1}{|\Phi_N|} \sum_{g \in \Phi_N}{\tilde{T}_g F} \in L^2 \left( W \times H^2 \right).
	\end{equation*}
	We want to show $A_N \to 0$ in $L^2 \left( W \times H^2 \right)$.
	Decompose the Haar measure on $W$ as $m_W = \int_Z{m_z~dz}$.
	Then
	\begin{align*}
		\norm{L^2 \left( W \times H^2 \right)}{A_N}^2 & = \int_W{\int_{H^2}{|A_N(w,x)|^2~dx}~dw} \\
		 & = \int_Z{\left( \int_{W_z}{\int_{H^2}{|A_N(w,x)|^2~dx}~dm_z} \right)~dz} \\
		 & = \int_Z{\norm{L^2 \left( W_z \times H^2 \right)}{A_N}^2~dz} \\
		 & \tendsto{N \to \infty} 0.
	\end{align*}
	
	To complete the proof, we apply the van der Corput trick.
	Let $u_g := T_{\varphi(g)}f_1 \cdot T_{\psi(g)}f_2 \in L^2(Z \times H)$.
	Then
	\begin{equation*}
		\innprod{u_{g+h}}{u_g} = \int_{Z \times H}{(\overline{f}_1 \cdot T_{\varphi(h)}f_1) \cdot T_{(\psi - \varphi)(g)}(\overline{f}_2 \cdot T_{\psi(h)}f_2)~d\mu}.
	\end{equation*}
	Since $(\psi - \varphi)(\Gamma)$ has finite index in $\Gamma$, we have $\mI_{\psi - \varphi} \subseteq \mZ$,
	so by the mean ergodic theorem,
	\begin{align*}
		\xi_h & := \UClim_{g \in \Gamma}{\innprod{u_{g+h}}{u_g}} \\
		 & = \int_{Z \times H}{(\overline{f}_1 \cdot T_{\varphi(h)}f_1)(z,x) \cdot
		 \left( \int_Z{\E{\overline{f}_2 \cdot T_{\psi(h)}f_2}{\mZ}
		 \left( z + \left( \hat{\psi} - \hat{\varphi} \right)(t) \right)}~dt \right)~dz~dx} \\
		 & = \int_W{\E{\overline{f}_1 \cdot T_{\varphi(h)}f_1}{\mZ}(w_1)
		 \cdot \E{\overline{f}_2 \cdot T_{\psi(h)}f_2}{\mZ}(w_2)~dw} \\
		 & = \int_{W \times H^2}{\left( \overline{f}_1 \cdot T_{\varphi(h)}f_1 \right)(w_1, x_1)
		 \left( \overline{f}_2 \cdot T_{\psi(h)}f_2 \right)(w_2, x_2)~dw~dx} \\
		 & = \innprod{\tilde{T}_hF}{F}_{L^2 \left( W \times H^2 \right)}.
	\end{align*}
	Now since $\UClim_{h \in \Gamma}{\tilde{T}_hF} = 0$ in $L^2(W \times H^2)$,
	we have $\UClim_{h \in \Gamma}{\xi_h} = 0$.
	By the van der Corput lemma (Lemma \ref{lem: vdC}),
	it follows that $\UClim_{g \in \Gamma}{u_g} = 0$ in $L^2(Z \times H)$ as desired.
\end{proof}

The next step is to construct the function $\omega$ appearing in Theorem \ref{thm: limit}.
Fix $z \in Z$ such that $M_z = M$.
Note that our assumptions on the discrete spectrum of $\bfX$ ensure that $(W_z, \tilde{\alpha})$ is isomorphic to $(Z, \alpha)$
under the map $(w_1, w_2) \mapsto \hat{\theta}_1(w_1 - z) + \hat{\theta}_2(w_2 - z)$.
Hence, by Proposition \ref{prop: continuous extension mod Mackey},
there is a function $\tilde{\omega} : Z \times Z \to H^2/M$
such that $t \mapsto \tilde{\omega}(t, \cdot)$ is a continuous map $Z \to \M(Z, H^2/M)$ and
\begin{equation*}
	\tilde{\omega}(\alpha_g, z) \equiv \tilde{\sigma}_g(z) \pmod{M}.
\end{equation*}
By the Kuratowski and Ryll-Nardewski measurable selection theorem (see \cite[Section 5.2]{srivastava}),
we can lift $\tilde{\omega}$ to a function $\omega : Z \times Z \to H^2$ such that $\omega + M = \tilde{\omega}$. \\

Now we put everything together to prove Theorem \ref{thm: limit}:

\begin{proof}[Proof of Theorem \ref{thm: limit}]
	By linearity, we may assume $f_i = h_i \otimes \chi_i$ with $\chi_i \in \hat{H}$ for $i = 1, 2$.
	Let $\tilde{\chi} = \chi_1 \otimes \chi_2 \in \hat{H^2}$.
	Then the right hand side of \eqref{eq: limit} is equal to
	\begin{equation} \label{eq: right hand side 2}
		\chi_1(x) \chi_2(x) \left( \int_Z{h_1 \left( z + \hat{\varphi}(t) \right) h_2 \left( z + \hat{\psi}(t) \right) \tilde{\chi} \left( \omega(t,z) \right)~dt} \right) \left( \int_M{\tilde{\chi}~dm_M} \right).
	\end{equation}
	We break the proof into two cases depending on whether or not $\tilde{\chi}$ belongs to $M^{\perp}$. \\
	
	\underline{Case 1}: $\tilde{\chi} \notin M^{\perp}$.
	
	In this case, the expression in \eqref{eq: right hand side 2} is equal to 0.
	Proposition \ref{prop: Mackey orthogonality} guarantees that the left hand side of \eqref{eq: limit} is also equal to zero. \\
	
	\underline{Case 2}: $\tilde{\chi} \in M^{\perp}$.
	
	Let $F : Z \times Z \to \C$ be given by
	\begin{equation*}
		F(t, z) = h_1 \left( z + \hat{\varphi}(t) \right) h_2 \left( z + \hat{\psi}(t) \right) \tilde{\chi} \left( \omega(t,z) \right).
	\end{equation*}
	Note that $t \mapsto F(t, \cdot)$ is a continuous function from $Z$ to $L^2(Z)$,
	and the expression in \eqref{eq: right hand side 2} simplifies to $\chi_1(x) \chi_2(x) \int_Z{F(t,z)~dt}$.
	
	Moving to the left hand side of \eqref{eq: limit}, for $g \in \Gamma$, we have
	\begin{equation*}
		f_1(T_{\varphi(g)}(z,x)) f_2(T_{\psi(g)}(z,x)) = \chi_1(x) \chi_2(x) F(\alpha_g, z)
	\end{equation*}
	Therefore, since $(Z, \alpha)$ is uniquely ergodic, we have
	\begin{equation*}
		\UClim_{g \in \Gamma}{F(\alpha_g, z)} = \int_Z{F(t,z)~dt}
	\end{equation*}
	in $L^2(Z)$ by Lemma \ref{lem: Hil-valued ET}.
	This completes the proof.
\end{proof}


\section{Khintchine-type recurrence} \label{sec: proof}

In this section, we prove Theorem \ref{thm: main}.
We will first prove the following statement (corresponding to the case $\eta = \id_{\Gamma}$),
from which Theorem \ref{thm: main} can be quickly deduced:

\begin{thm} \label{thm: averaging form}
	Let $\Gamma$ be a countable discrete abelian group.
	Let $\varphi, \psi \in \End(\Gamma)$ such that $(\psi - \varphi)(\Gamma)$ is a finite index subgroup of $\Gamma$.
	Suppose there exist $\theta_1, \theta_2 \in \End(\Gamma)$ such that
	$\theta_1 \circ \varphi + \theta_2 \circ \psi$ is injective.
	Then for any ergodic measure-preserving $\Gamma$-system $\left( X, \mX, \mu, (T_g)_{g \in \Gamma} \right)$
	with Kronecker factor $\bfZ = (Z, \alpha)$, any $A \in \mX$, and any $\eps > 0$,
	there is a continuous function $\kappa : Z \to [0, \infty)$ with $\int_Z{\kappa(z)~dz} = 1$ such that
	\begin{equation*}
		\UClim_{g \in \Gamma}{\kappa(\alpha_g)~\mu \left( A \cap T_{\varphi(g)}^{-1}A \cap T_{\psi(g)}^{-1}A \right)} > \mu(A)^3 - \eps.
	\end{equation*}
\end{thm}
\begin{proof}
	Let $\bfX = \left( X, \mX, \mu, (T_g)_{g \in \Gamma} \right)$ be an ergodic $\Gamma$-system
	with Kronecker factor $\bfZ = (Z, \alpha)$, let $A \in \mX$, and let $\eps > 0$.
	Put $f = \ind_A$.
	Let $\kappa : Z \to [0, \infty)$ be a continuous function (to be specified later) with $\int_Z{\kappa(z)~dz} = 1$.
	By \cite[Theorem 4.10]{abs}, we may assume without loss of generality (by passing to an extension if necessary) that
	\begin{align*}
		\UClim_{g \in \Gamma}&{~\kappa(\alpha_g)~\int_X{f \cdot T_{\varphi(g)} f \cdot T_{\psi(g)} f~d\mu}} \\
		 & = \UClim_{g \in \Gamma}{\kappa(\alpha_g)~\int_X{f \cdot T_{\varphi(g)} \E{f}{\mZ^2 \lor \mI_{\varphi}}
		 \cdot T_{\psi(g)} \E{f}{\mZ^2 \lor \mI_{\psi}}~d\mu}}.
	\end{align*}
	We may further assume that the discrete spectrum of $\bfX$ is $\{\varphi, \psi, \theta_1, \theta_2\}$-complete
	and $(\theta_1 \circ \varphi + \theta_2 \circ \psi)$-divisible
	by passing to another extension if necessary with the help of Theorem \ref{thm: complete divisible extension}.
	
	Let $f_{\varphi} = \E{f}{\mZ^2 \lor \mI_{\varphi}}$ and $f_{\psi} = \E{f}{\mZ^2 \lor \mI_{\psi}}$.
	We may expand
	\begin{equation*}
		f_{\varphi} = \sum_{i}{c_i h_i} \qquad \text{and} \qquad f_{\psi} = \sum_{j}{d_j k_j},
	\end{equation*}
	where $c_i$ is $\varphi(\Gamma)$-invariant, $d_j$ is $\psi(\Gamma)$-invariant, and $h_i$, $k_j$ are $\mZ^2$-measurable.
	Write $\bfZ^2 = \bfZ \times_{\sigma} H$ and apply Corollary \ref{cor: twisted limit}:
	\begin{align*}
		\UClim_{g \in \Gamma}&{~\kappa(\alpha_g)~\mu \left( A \cap T_{\varphi(g)}^{-1}A \cap T_{\psi(g)}^{-1}A \right)} \\
		 & = \UClim_{g \in \Gamma}{\kappa(\alpha_g)~\sum_{i,j}
		 {\int_X{f c_i d_j \cdot T_{\varphi(g)} h_i \cdot T_{\psi(g)} k_j~d\mu}}} \\
		 & = \sum_{i,j}{\int_{X \times Z \times M}{f(x) c_i(x) d_j(x) \cdot \kappa(t)
		 \cdot h_i \left( \pi_Z(x) + \hat{\varphi}(t), \pi_H(x) + u + \omega_1 \left( t, \pi_Z(x) \right) \right)}} \\
		 & \qquad \qquad \qquad {{ k_j \left( \pi_Z(x) + \hat{\psi}(t), \pi_H(x) + v + \omega_2 \left( t, \pi_Z(x) \right) \right)
		 ~d\mu(x)~dm_Z(t)~dm_M(u,v)}}.
	\end{align*}
	Taking $\kappa$ supported on a sufficiently small neighborhood of $0$ in $Z$, it suffices to establish the inequality
	\begin{equation} \label{eq: key inequality}
		\sum_{i,j}{\int_{X \times M}{f(x) c_i(x) d_j(x) h_i \left( \pi_Z(x), \pi_H(x) + u \right) k_j \left( \pi_Z(x), \pi_H(x) + v \right)~d\mu(x)~dm_M(u,v)}} \ge \mu(A)^3.
	\end{equation}
	
	By applying Theorem \ref{thm: Mackey prod} together with Theorem \ref{thm: HK extension}
	and passing to yet another extension,
	we may assume the Mackey group $M$ decomposes as $M = M_{\varphi} \times M_{\psi}$.
	Let $\mW_{\varphi}$ be the $\sigma$-algebra generated by functions $f : Z \times H \to \C$ satisfying $f(z,x) = f(z,x+y)$
	for $y \in M_{\varphi}$, and let $\mW_{\psi}$ be defined similarly.
	Then the left hand side of \eqref{eq: key inequality} is equal to
	\begin{equation*}
		\int_X{f \cdot \E{f}{\mW_{\varphi} \lor \mI_{\varphi}} \cdot \E{f}{\mW_{\psi} \lor \mI_{\psi}}~d\mu},
	\end{equation*}
	which is bounded below by $\left( \int_X{f~d\mu} \right)^3 = \mu(A)^3$ by \cite[Lemma 1.6]{chu}.
\end{proof}

Now we complete the proof of Theorem \ref{thm: main}:

\begin{proof}[Proof of Theorem \ref{thm: main}]
	Let $\bfX = \left( X, \mX, \mu, (T_g)_{g \in \Gamma} \right)$ be an ergodic measure-preserving $\Gamma$-system,
	let $A \in \mX$, and let $\eps > 0$.
	
	The $\Gamma$-system $\left( X, \mX, \mu, \left( T_{\eta(g)} \right)_{g \in \Gamma} \right)$
	has finitely many ergodic components (at most the index of $\eta(\Gamma)$ in $\Gamma$).
	Let $\mu = \sum_{i=1}^l{\mu_i}$ be the ergodic decomposition.
	Then each of the systems $\bfX_i = \left( X, \mX, \mu_i, \left( T_{\eta(g)} \right)_{g \in \Gamma} \right)$ is ergodic,
	and they all have the same Kronecker factor $\bfZ_{\eta} = (Z_{\eta}, \alpha \circ \eta)$,
	where $Z_{\eta} = \overline{\left\{ \alpha_{\eta(g)} : g \in \Gamma \right\}}$.
	By Theorem \ref{thm: averaging form}, we may therefore find a continuous function
	$\kappa : Z_{\eta} \to [0, \infty)$ such that $\int_{Z_{\eta}}{\kappa(t)~dt} = 1$ and
	\begin{equation*}
		\UClim_{h \in \eta(\Gamma)}{\kappa \left( \alpha_h \right) \mu_i \left( A \cap T_{\varphi'(h)}^{-1} A \cap T_{\psi'(h)}^{-1} A \right)} > \mu_i(A)^3 - \eps
	\end{equation*}
	for each $i \in \{1, \dots, l\}$.
	Then by Jensen's inequality,
	\begin{equation*}
		\UClim_{h \in \eta(\Gamma)}{\kappa \left( \alpha_h \right) \mu \left( A \cap T_{\varphi'(h)}^{-1} A \cap T_{\psi'(h)}^{-1} A \right)} > \mu(A)^3 - \eps.
	\end{equation*}
	It follows that
	\begin{equation*}
		\left\{ h \in \eta(\Gamma) : \mu \left( A \cap T_{\varphi'(h)}^{-1} A \cap T_{\psi'(h)}^{-1} A \right) > \mu(A)^3 - \eps \right\}
	\end{equation*}
	is syndetic in $\eta(\Gamma)$:
	if not, then by \cite[Lemma 1.9]{abb}, there exists a F{\o}lner sequence $(\Phi_N)_{N \in \N}$ in $\eta(\Gamma)$ such that
	\begin{equation*}
		\mu \left( A \cap T_{\varphi'(h)}^{-1} A \cap T_{\psi'(h)}^{-1} A \right) \le \mu(A)^3 - \eps
	\end{equation*}
	for every $h \in \bigcup_{N \in \N}{\Phi_N}$, whence
	\begin{equation*}
		\lim_{N \to \infty}{\frac{1}{|\Phi_N|} \sum_{h \in \Phi_N}{\kappa \left( \alpha_h \right) \mu \left( A \cap T_{\varphi'(h)}^{-1} A \cap T_{\psi'(h)}^{-1} A \right)}} \le \left( \mu(A)^3 - \eps \right) \int_{Z_{\eta}}{\kappa(t)~dt} = \mu(A)^3 - \eps,
	\end{equation*}
	which is a contradicition.
	But $\eta(\Gamma)$ has finite index in $\Gamma$, so
	\begin{equation*}
		\left\{ g \in \Gamma : \mu \left( A \cap T_{\varphi(g)}^{-1} A \cap T_{\psi(g)}^{-1} A \right) > \mu(A)^3 - \eps \right\}
	\end{equation*}
	is syndetic in $\Gamma$.
\end{proof}


\section*{Acknowledgements}

This material is based upon work supported by the National Science Foundation under Grant No. DMS-1926686.
Thanks to Vitaly Bergelson for helpful comments on the introduction and to Or Shalom for several useful discussions.


\end{document}